\documentclass[a4paper]{amsart}

\usepackage{amscd,amsthm,amsfonts,amssymb,esint}
\usepackage{amsmath}
\usepackage[colorlinks=true]{hyperref}
\usepackage{fullpage}
\usepackage[all]{xy}
\usepackage{tikz-cd}

    \newcommand{\BA}{{\mathbb {A}}} 
    \newcommand{\BC}{{\mathbb {C}}} 
     \newcommand{\BF}{{\mathbb {F}}}
     
    \newcommand{\BI}{{\mathbb {I}}}

    \newcommand{\BQ}{{\mathbb {Q}}}

     \newcommand{\BZ}{{\mathbb {Z}}}

    \newcommand{\CC}{{\mathcal {C}}} 
     \newcommand{\CF}{{\mathcal {F}}}
     \newcommand{\CH}{{\mathcal {H}}}
     
     \newcommand{\CL}{{\mathcal {L}}}
     
    \newcommand{\CO}{{\mathcal {O}}}

     \newcommand{\CX}{{\mathcal {X}}}

     \newcommand{\ff}{{\mathfrak{f}}}
    \newcommand{\fg}{{\mathfrak{g}}}

    \newcommand{\fm}{{\mathfrak{m}}} 
     \newcommand{\fp}{{\mathfrak{p}}}
    \newcommand{\fq}{{\mathfrak{q}}}

     \newcommand{\fX}{{\mathfrak{X}}}

    \newcommand{\Aut}{{\mathrm{Aut}}}

    \newcommand{\Char}{{\mathrm{Char}}}

    \newcommand{\coker}{{\mathrm{coker}}}

    \newcommand{\Frob}{{\mathrm{Frob}}}

    \newcommand{\Gal}{{\mathrm{Gal}}} \newcommand{\GL}{{\mathrm{GL}}}
    
    \newcommand{\Hom}{{\mathrm{Hom}}}
    
    \newcommand{\id}{{\mathrm{id}}}
    \renewcommand{\Im}{{\mathrm{Im~}}}

    \newcommand{\irred}{{\mathrm{irred}}}

    \newcommand{\ord}{{\mathrm{ord}}}

    \renewcommand{\mod}{\ \mathrm{mod}\ }

    \newcommand{\rel}{{\mathrm{rel}}}

    \newcommand{\Sel}{{\mathrm{Sel}}} 
    
    \newcommand{\str}{{\mathrm{str}}}

    \newcommand{\sign}{{\mathrm{sign}}}
    \newcommand{\SL}{{\mathrm{SL}}}

    \newcommand{\tr}{{\mathrm{tr}}}\newcommand{\tor}{{\mathrm{tor}}}
    
    \newcommand{\ur}{{\mathrm{ur}}}

 \newcommand{\Nm}{{\mathrm{Nm}}}

    \font\cyr=wncyr10

    \newcommand{\Sha}{\hbox{\cyr X}}

    \newcommand{\ov}{\overline}

    \newcommand{\ra}{\rightarrow}

    \theoremstyle{plain}
    \newtheorem{thm}{Theorem}[section] \newtheorem{cor}[thm]{Corollary}
    \newtheorem{lem}[thm]{Lemma}  \newtheorem{prop}[thm]{Proposition}
    \newtheorem {conj}[thm]{Conjecture} \newtheorem{defn}[thm]{Definition}

\theoremstyle{remark} \newtheorem{remark}[thm]{Remark}
\theoremstyle{remark} 
\theoremstyle{remark} 

    \newcommand{\Neron}{N\'{e}ron~}

    \numberwithin{equation}{section}

    \newcommand{\Gr}{{\mathrm{Gr}}}
    
\begin{document}

\title{Main conjectures for non-CM elliptic curves at good ordinary primes}
\author{Xiaojun Yan}
\address{Academy of Mathematics and Systems Science,
Chinese Academy of Sciences,
Beijing
100190, China.}
\email{xjyan95@amss.ac.cn}

\author{Xiuwu Zhu}
\address{Beijing Institute of Mathematical Sciences and Applications, Beijing 101408, China;}
\address{Yau Mathematical Sciences Center, Tsinghua University, Beijing 100084, China.}
\email{xwzhu@bimsa.cn}

\begin{abstract}
Let $E/\mathbb{Q}$ be an elliptic curve, $K$ an imaginary quadratic field, and let $p > 2$ be a prime that splits in $K$ and at which $E$ has good ordinary reduction. 
Assume that the residual Galois representation associated with $(E, p)$ is irreducible. 
In this paper, we establish new cases of the two-variable Iwasawa main conjecture for $E$ over $K$. 
As applications, we obtain more general results on the $p$-converse theorem and the $p$-part of the Birch and Swinnerton-Dyer formula in rank at most one.

\end{abstract}

\maketitle
\tableofcontents
\section{Introduction}
Let $E$ be an elliptic curve over $\mathbb{Q}$. 
Fix an odd prime $p$ and embeddings $\iota_p:\overline{\mathbb{Q}} \hookrightarrow \overline{\mathbb{Q}}_p$, $\iota_\infty:\overline{\mathbb{Q}} \hookrightarrow \mathbb{C}$. 
Let $T_pE$ denote the $p$-adic Tate module of $E$, $\rho_E: G_{\mathbb{Q}} \to \operatorname{Aut}_{\mathbb{Z}_p}(T_pE)$ the associated $p$-adic Galois representation, and $\ov{\rho}_E:G_{\mathbb{Q}} \to \operatorname{Aut}(E[p])$ the corresponding residual representation, i.e., the reduction modulo $p$ of $\rho_E$.

To study the arithmetic properties of $E$, such as those predicted by the Birch and Swinnerton-Dyer conjecture, one often turns to tools from Iwasawa theory. 
There are two central objects in Iwasawa theory: Selmer groups over Iwasawa extensions and $p$-adic $L$-functions. 
The Iwasawa main conjectures establish a deep connection between them.

Let $K$ be an imaginary quadratic field. 
Suppose that $p = \mathfrak{p} \bar{\mathfrak{p}}$ splits in $K$, where $\mathfrak{p}$ is the prime determined by the embedding $\iota_p$. 
Let $K_\infty^+$ and $K_\infty^-$ denote the cyclotomic and anticyclotomic $\mathbb{Z}_p$-extensions of $K$, respectively, and set $K_\infty = K_\infty^+ K_\infty^-$. 
Then the associated Iwasawa algebras are given by
\[
\Lambda_K^{\pm} = \mathbb{Z}_p[[\operatorname{Gal}(K_\infty^{\pm}/K)]], \quad 
\Lambda_K = \mathbb{Z}_p[[\operatorname{Gal}(K_\infty/K)]].
\]
For any locally compact $\mathbb{Z}_p$-module $M$, let $M^\vee := \operatorname{Hom}_{\mathrm{cont}}(M, \mathbb{Q}_p/\mathbb{Z}_p)$ denote its Pontryagin dual.

We consider the following two types of Selmer groups over $K_\infty$:
\begin{enumerate}
    \item the $\Lambda_K$-module Selmer group $H^1_{\mathcal{F}_{\mathrm{ord}}}(K, T_pE \otimes \Lambda_K^\vee)$, with ordinary local conditions at $v \mid p$;
    \item the $\Lambda_K$-module Selmer group $H^1_{\mathcal{F}_{\mathrm{Gr}}}(K, T_pE \otimes \Lambda_K^\vee)$, with relaxed local conditions at $v = \mathfrak{p}$ and strict conditions at $v = \bar{\mathfrak{p}}$.
\end{enumerate}

As in Castella-Grossi-Skinner \cite{CGS}, we have two versions of the two-variable $p$-adic $L$-function:
\begin{enumerate}
    \item $\mathcal{L}_p^{\mathrm{PR}}(E/K) \in \Lambda_K$, which interpolates the algebraic parts of $L(E/K, \chi^{-1}, 1)$ for finite order characters $\chi$ of $\operatorname{Gal}(K_\infty/K)$;
    \item $\mathcal{L}_p^{\mathrm{Gr}}(E/K) \in \Lambda_K^{\mathrm{ur}} := \Lambda_K \widehat{\otimes}_{\mathbb{Z}_p} \mathbb{Z}_p^{\mathrm{ur}}$, which interpolates the algebraic parts of $L(E/K, \chi, 1)$ for characters $\chi$ of $\Gal(K_\infty/K)$ with infinity type $(b, a)$ such that $a \leq -1$ and $b \geq 1$, where $\mathbb{Z}_p^{\mathrm{ur}}$ denotes the ring of integers of the maximal unramified extension of $\mathbb{Q}_p$.
\end{enumerate}

\begin{conj}\label{mainconj}

    Suppose that the residual representation $\bar{\rho}_E|_{G_K}$ is irreducible.
    \begin{enumerate}
        \item $H^1_{\mathcal{F}_{\mathrm{ord}}}(K, T_pE \otimes \Lambda_K^\vee)^\vee$ is a torsion $\Lambda_K$-module and
        \[
        \operatorname{Char}_{\Lambda_K} \left(H^1_{\mathcal{F}_{\mathrm{ord}}}(K, T_pE \otimes \Lambda_K^\vee)^\vee \right) = \left( \mathcal{L}_p^{\mathrm{PR}}(E/K) \right).
        \]
        \item $H^1_{\mathcal{F}_{\mathrm{Gr}}}(K, T_pE \otimes \Lambda_K^\vee)^\vee$ is a torsion $\Lambda_K$-module and
        \[
        \operatorname{Char}_{\Lambda_K} \left(H^1_{\mathcal{F}_{\mathrm{Gr}}}(K, T_pE \otimes \Lambda_K^\vee)^\vee \right) \Lambda_K^{\mathrm{ur}} = \left( \mathcal{L}_p^{\mathrm{Gr}}(E/K) \right).
        \]
    \end{enumerate}
\end{conj}

\subsection{Main result}

\begin{thm}\label{mainthm}
    Suppose that the Heegner hypothesis holds (in particular, $\operatorname{sign}(E/K) = -1$), and that $\bar{\rho}_E|_{G_K}$ is absolutely irreducible. Then:
    \begin{enumerate}
        \item $H^1_{\mathcal{F}_{\mathrm{ord}}}(K, T_pE \otimes \Lambda_K^\vee)^\vee$ is a torsion $\Lambda_K$-module, and
        \[
        \operatorname{Char}_{\Lambda_K}\left( H^1_{\mathcal{F}_{\mathrm{ord}}}(K, T_pE \otimes \Lambda_K^\vee)^\vee \right) \subset \left( \mathcal{L}_p^{\mathrm{PR}}(E/K) \right).
        \]
        \item $H^1_{\mathcal{F}_{\mathrm{Gr}}}(K, T_pE \otimes \Lambda_K^\vee)^\vee$ is a torsion $\Lambda_K$-module, and
        \[
        \operatorname{Char}_{\Lambda_K}\left( H^1_{\mathcal{F}_{\mathrm{Gr}}}(K, T_pE \otimes \Lambda_K^\vee)^\vee \right)\Lambda_K^{\mathrm{ur}} \subset \left( \mathcal{L}_p^{\mathrm{Gr}}(E/K) \right).
        \]
    \end{enumerate}
    Moreover, if 
    \begin{equation}\tag{Im}\label{Imag}
        \text{there exists } \tau \in \operatorname{Gal}(\overline{\mathbb{Q}}/\mathbb{Q}(\mu_{p^\infty})) \text{ such that } T_pE / (\rho_E(\tau) - 1)T_pE \text{ is free of rank one over } \BZ_p
    \end{equation}
    then Conjecture~\ref{mainconj} holds.
\end{thm}

Skinner-Urban \cite{skinner2014iwasawa} first proved Conjecture~\ref{mainconj}(1) using Eisenstein congruences on $GU(2,2)$ under certain assumptions, in particular when $\operatorname{sign}(E/K) = +1$.
Wan \cite{Wan1,Wan} proved Conjecture~\ref{mainconj}(2) via Eisenstein congruences on $GU(3,1)$ under different assumptions, specifically when $\operatorname{sign}(E/K) = -1$ and $E$ is semistable.

In \cite{BSTW}, Burungale-Skinner-Tian-Wan established an equivalence between parts (1) and (2) of Conjecture~\ref{mainconj} by employing Beilinson-Flach elements and an explicit reciprocity law. This, combined with the result of Skinner-Urban concerning  Conjecture~\ref{mainconj}(1), enabled them to prove new cases of  Conjecture~\ref{mainconj}(2).

Our proof of Theorem~\ref{mainthm} adopts an approach similar to that presented in \cite{BSTW}. A key observation is that, under more general conditions, Skinner-Urban in fact proved that the left-hand side of Conjecture~\ref{mainconj}(1) is contained in the right-hand side after tensoring with the field of fractions of $\Lambda_K^+$.
To achieve full equality, they applied Vatsal's non-vanishing result \cite{Vat} on the $\mu$-invariant of the anticyclotomic projection of $\mathcal{L}_p^{\mathrm{PR}}(E/K)$ to show that $\operatorname{ord}_P(\mathcal{L}_p^{\mathrm{PR}}(E/K)) = 0$ for height-one primes $P$ of $\Lambda_K$ coming from $\Lambda_K^+$. This required the following conditions:
\begin{enumerate}
    \item $N = N^+N^-$, where $N^+$ is divisible only by primes that split in $K$ and $N^-$ is a square-free product of an odd number of primes that are inert in $K$, implying $\operatorname{sign}(E/K) = +1$;
    \item For every $\ell \mid N^-$, the residual representation satisfies $\bar{\rho}_E|_{I_\ell} \not\equiv 1$, where $I_\ell$ is the inertia subgroup at $\ell$.
\end{enumerate}
In particular, this forces $E$ to have a bad semistable prime.

In this paper, using the equivalence in \cite{BSTW}, we translate Skinner-Urban's results on Conjecture~\ref{mainconj}(1) into the setting of Conjecture~\ref{mainconj}(2). This allows us to conclude that the left-hand side of Conjecture~\ref{mainconj}(2) is contained in the right-hand side after tensoring with the fractional field of $\Lambda_K^+$. 

To obtain a full inclusion, we apply Hsieh's result \cite{hsieh2014special} on the non-vanishing of the $\mu$-invariant of the anticyclotomic projection of $\mathcal{L}_p^{\mathrm{Gr}}(E/K)$ in the case $\operatorname{sign}(E/K) = -1$. Since we work under the Heegner hypothesis, we can avoid the assumption that $E$ has a bad semistable prime.

As in \cite{skinner2014iwasawa}, we then invoke Kato's theorem on Mazur's main conjecture \cite{Kato} to complete the proof of Theorem~\ref{mainthm}.

\begin{remark}
    Recently, using Wan's extension of Skinner-Urban's method to the setting of Hilbert modular forms \cite{wan14}, together with a base change argument, Burungale-Castella-Skinner \cite{burungale2024base} proved the rational version of Conjecture~\ref{mainconj} (i.e., equality after tensoring with $\mathbb{Q}_p$) under the assumptions that $p > 3$ and $\bar{\rho}_E$ is irreducible, and proved the integral version when condition \eqref{Imag} also holds.

    Roughly speaking, by choosing a suitable real quadratic field $F$ and a CM extension $K'/F$ with $[K':F] = 2$, they study a three-variable main conjecture for $E_F/K'$ analogous to Conjecture~\ref{mainconj}(1). Following Skinner-Urban's method, Wan proves one direction of the divisibility result, similar to \cite{skinner2014iwasawa}, and applies a non-vanishing theorem for the $\mu$-invariant of an anticyclotomic $p$-adic $L$-function over totally real fields \cite{Hun}.

    In their setup, the analogue of $N^-$ is a square-free product of an even number of primes of $F$ inert in $K'$, which in particular may be trivial. Thus, they can derive results on various main conjectures for $E$ without requiring $E$ to have bad semistable reduction. Their arguments require $p > 3$ in order to apply \cite{Hun}, but do not require $(E, K)$ to satisfy the Heegner hypothesis.
\end{remark}

As an application, we prove additional cases of one-variable main conjectures, the $p$-part of the Birch and Swinnerton-Dyer formula, and the $p$-converse theorem.

\begin{cor}
Let $E/\mathbb{Q}$ be an elliptic curve of conductor $N$, and let $p \nmid 2N$ be a prime at which $E$ has good ordinary reduction. If $r \leq 1$, then the following are equivalent:
\begin{enumerate}
    \item $\operatorname{corank}_{\mathbb{Z}_p} \operatorname{Sel}_{p^\infty}(E/\mathbb{Q}) = r$;
    \item $\operatorname{ord}_{s=1} L(E/\mathbb{Q}, s) = r$.
\end{enumerate}
Moreover, under either condition, if \eqref{Imag} also holds, then the $p$-part of the Birch and Swinnerton-Dyer formula for $E$ holds.
\end{cor}

\begin{remark}
There has been extensive work on the $p$-part of the Birch and Swinnerton-Dyer formula and the $p$-converse theorem for elliptic curves. In our context, see \cite{skinner2014iwasawa}, \cite{zhang2014selmer}, \cite{JSW}, \cite{BSTW}, \cite{burungale2024base} for earlier results on the $p$-part of the BSD formula, and \cite{skinner2020converse}, \cite{zhang2014selmer}, \cite{BSTW}, \cite{burungale2023non} for related results on the $p$-converse theorem.
\end{remark}

\subsection{Strategy}
By \cite[Proposition 4.2.1]{CGS} or \cite[Proposition 9.18]{BSTW}, the use of Beilinson-Flach elements, combined with an explicit reciprocity law, establishes a connection between various formulations of the main conjectures. 
As a result, the inclusion relations (and their converses) in Theorem~\ref{mainthm}(1) and (2) are shown to be equivalent.

From the work of Skinner-Urban \cite{skinner2014iwasawa}, especially Theorem 7.7 and Proposition 13.6(1), we have the following:

\begin{thm}\label{skinner2014iwasawa}
    Suppose that the residual representation $\bar{\rho}_{E}: G_{\mathbb{Q}} \to \operatorname{Aut}(E[p])$ is irreducible. Then for any height-one prime $P$ of $\Lambda_K$, we have
    \[
    \operatorname{ord}_P \left( \operatorname{Char}_{\Lambda_K} \left( H^1_{\mathcal{F}_{\mathrm{ord}}}(K, T_pE \otimes \Lambda_K^\vee)^\vee \right) \right)
    \geq \operatorname{ord}_P \left( \mathcal{L}_p^{\mathrm{PR}}(E/K) \right),
    \]
    unless $P = P^+ \Lambda_K$ for some height-one prime $P^+ \subset \Lambda_K^+$.
\end{thm}

Following \cite[Proposition 9.18]{BSTW} (or \cite[Proposition 4.2.1]{CGS}), we have:

\begin{thm}\label{Gr}
    Under the same assumptions as in the previous theorem, for any height-one prime $P$ of $\Lambda_K^{\mathrm{ur}}$, we have
    \[
    \operatorname{ord}_P \left( \operatorname{Char}_{\Lambda_K} \left( H^1_{\mathcal{F}_{\mathrm{Gr}}}(K, T_pE \otimes \Lambda_K^\vee)^\vee \right) \Lambda_K^{\mathrm{ur}} \right)
    \geq \operatorname{ord}_P \left( \mathcal{L}_p^{\mathrm{Gr}}(E/K) \right),
    \]
    unless $P = P^+ \Lambda_K^{\mathrm{ur}}$ for some height-one prime $P^+ \subset \Lambda_K^{\mathrm{ur},+}$.
\end{thm}

However, by the result on the $\mu$-invariant of the BDP $p$-adic $L$-function due to Hsieh \cite{hsieh2014special}, if $P \subset \Lambda_K^{\mathrm{ur}}$ is a height-one prime of the form
$P = P^+ \Lambda_K^{\mathrm{ur}}$ for some $P^+ \subset \Lambda_K^{\mathrm{ur},+}$, then
\[
\operatorname{ord}_P \left( \mathcal{L}_p^{\mathrm{Gr}}(E/K) \right) = 0.
\]
Therefore, Theorem~\ref{mainthm}(2) holds, which implies that Theorem~\ref{mainthm}(1) also holds. 
Moreover, if condition \eqref{Imag} is satisfied, then as in \cite[Theorem 3.30]{skinner2014iwasawa}, we conclude that Conjecture~\ref{mainconj}(2) holds, and hence so does Conjecture~\ref{mainconj}(1).

\subsection{Acknowledgment}
The authors are grateful to Ye Tian and Wei He for their continuous encouragement and numerous helpful discussions, and to Ashay Burungale for his valuable comments.
We are also grateful to the reviewer for the thoughtful suggestions and invested expertise, which were instrumental in strengthening the paper.

\section{Selmer groups}
Let $K$ be an imaginary quadratic field with discriminant $D_K$, and let $p > 2$ be a prime that splits in $K$, say $p\mathcal{O}_K = \mathfrak{p} \bar{\mathfrak{p}}$. 
Let $\mathbb{Q}_\infty$ denote the cyclotomic $\mathbb{Z}_p$-extension of $\mathbb{Q}$, and let $K_\infty$ be the $\mathbb{Z}_p^2$-extension of $K$. We write $K_\infty^+$ and $K_\infty^-$ for the cyclotomic and anticyclotomic $\mathbb{Z}_p$-extensions of $K$, respectively.

Define the Galois groups
\[
\Gamma_{\mathbb{Q}} := \operatorname{Gal}(\mathbb{Q}_\infty / \mathbb{Q}), \quad 
\Gamma_K := \operatorname{Gal}(K_\infty / K), \quad 
\Gamma_K^\pm := \operatorname{Gal}(K_\infty^\pm / K).
\]
We identify $\Gamma_K^+ = \operatorname{Gal}(K_\infty^+ / K)$ with $\Gamma_{\mathbb{Q}} = \operatorname{Gal}(\mathbb{Q}_\infty / \mathbb{Q})$. Let $\gamma^\pm$ be a topological generator of $\Gamma_K^\pm$.

Let
\[
\Lambda_{\mathbb{Q}} := \mathbb{Z}_p[[\Gamma_{\mathbb{Q}}]], \quad 
\Lambda_K := \mathbb{Z}_p[[\Gamma_K]], \quad 
\Lambda_K^\pm := \mathbb{Z}_p[[\Gamma_K^\pm]]
\]
be the associated Iwasawa algebras. These rings are equipped with natural characters
\[
\varepsilon_{\mathbb{Q}}: G_{\mathbb{Q}} \twoheadrightarrow \Gamma_{\mathbb{Q}} \hookrightarrow \Lambda_{\mathbb{Q}}^\times, \quad 
\varepsilon_K: G_K \twoheadrightarrow \Gamma_K \hookrightarrow \Lambda_K^\times, \quad 
\varepsilon_{K,\pm}: G_K \twoheadrightarrow \Gamma_K^\pm \hookrightarrow \Lambda_K^{\pm,\times}
\]
arising from projection. The Pontryagin duals $\Lambda_{\mathbb{Q}}^\vee$, $\Lambda_K^\vee$, and $\Lambda_K^{\pm,\vee}$ are endowed with $G_{\mathbb{Q}}$- or $G_K$-actions via the inverses of these characters.

Throughout this paper, we normalize the reciprocity map of class field theory so that uniformizers correspond to arithmetic Frobenius elements. With this convention, we identify algebraic Hecke characters with their associated Galois characters.

\subsection{Selmer groups for Iwasawa algebras}
Let $E/\mathbb{Q}$ be an elliptic curve of conductor $N$ with $(N, D_K) = 1$. We assume that $E$ has good ordinary reduction at $p$.

\subsubsection*{Discrete Selmer groups}
Let $F$ be either $\mathbb{Q}$ or $K$. For a prime $w$ of $F$ above $p$, let $\mathbb{F}_w$ be the residue field at $w$, and let $\tilde{E}_{/\mathbb{F}_w}$ be the reduction of $E$. The kernel of the reduction map on the Tate module is denoted by $\mathcal{F}_w^+T_pE:=\ker\left(T_pE\rightarrow T_p\tilde{E}_{/\mathbb{F}_w}\right)$.

\begin{defn}\label{defn:local-condition}
    For $\Lambda$ any of $\Lambda_\BQ, \Lambda_K$, or $\Lambda_K^\pm$, and for a prime $w$ of $F$ over $p$, we define the following local conditions:
    \begin{enumerate}
       \item $H^1_{\ord}(F_w,T_pE\otimes\Lambda^\vee) := \operatorname{Im}\left(H^1(F_w,\mathcal{F}_w^+T_pE\otimes\Lambda^\vee)\rightarrow H^1(F_w,T_pE\otimes\Lambda^\vee)\right)$,
       
       \item $H^1_{\rel}(F_w,T_pE\otimes\Lambda^\vee) := H^1(F_w,T_pE\otimes\Lambda^\vee)$,
       
       \item $H^1_{\str}(F_w,T_pE\otimes\Lambda^\vee) := 0$.
    \end{enumerate}
\end{defn}

Let $\Sigma$ be a set of places of $F$ containing all places dividing $pN\infty$. In the anticyclotomic case ($F=K, \Lambda=\Lambda_K^-$), we assume moreover that every finite place in $\Sigma$ splits completely in $K/\mathbb{Q}$. Let $F_\Sigma$ be the maximal extension of $F$ unramified outside $\Sigma$, and let $G_{F,\Sigma}:=\Gal(F_\Sigma/F)$. 
 
In the case $F=\mathbb{Q}$, for $a\in\{\ord, \str, \rel\}$ and $M=T_pE\otimes\Lambda_\mathbb{Q}^\vee$, we define the Selmer group
\[H^1_{\mathcal{F}_{a}}(\mathbb{Q},M) := \ker\left(H^1(G_{\mathbb{Q},\Sigma},M)\rightarrow \prod_{q\in\Sigma, q\nmid p\infty} H^1(\mathbb{Q}_q,M)\times \frac{H^1(\mathbb{Q}_p,M)}{H^1_{a}(\mathbb{Q}_p,M)}\right).\]
Its Pontryagin dual is denoted by
\[\mathcal{X}_{\mathcal{F}_{a}}(E/\mathbb{Q}_\infty) := H^1_{\mathcal{F}_{a}}(\mathbb{Q},T_pE\otimes\Lambda_\mathbb{Q}^\vee)^\vee.\]
In the case $F=K$, for $a,b\in\{\ord, \str, \rel\}$, $\Lambda\in \{\Lambda_K,\Lambda_K^-,\Lambda_K^+\}$, and $M=T_pE\otimes\Lambda^\vee$, we define
\[H^1_{\mathcal{F}_{a,b}}(K,M) := \ker\left(H^1(G_{K,\Sigma},M)\rightarrow \prod_{\mathfrak{q}\in\Sigma,\mathfrak{q}\nmid p}H^1(K_\mathfrak{q},M)\times \frac{H^1(K_\mathfrak{p},M)}{H^1_{a}(K_\mathfrak{p},M)}\times \frac{H^1(K_{\bar{\mathfrak{p}}},M)}{H^1_{b}(K_{\bar{\mathfrak{p}}},M)}\right).\]
For simplicity, we write
\begin{enumerate}
    \item $H^1_{\mathcal{F}_a}(K,M) := H^1_{\mathcal{F}_{a,a}}(K,M)$ for $a\in\{\ord, \str, \rel\}$,
    \item $H^1_{\mathcal{F}_\Gr}(K,M) := H^1_{\mathcal{F}_{\rel,\str}}(K,M)$.
\end{enumerate}
For $a\in\{\ord, \str, \rel, \Gr\}$, we define the dual Selmer groups
\[\mathcal{X}_{\mathcal{F}_{a}}(E/K_\infty) := H^1_{\mathcal{F}_{a}}(K,T_pE\otimes\Lambda_K^\vee)^\vee,\quad \mathcal{X}_{\mathcal{F}_{a}}(E/K^\pm_\infty) := H^1_{\mathcal{F}_{a}}(K,T_pE\otimes\Lambda_K^{\pm,\vee})^\vee.\]

\begin{lem}\label{pcond}
Let $\Lambda$ be $\Lambda_K$, $\Lambda_\BQ$, or $\Lambda_K^{\pm}$.
    For every prime $\fq$ of $K$ dividing $p$, we have 
    $$H^1\left(G_{K_\fq}/I_\fq,(T_p\tilde{E}_{/\BF_\fq}\otimes\Lambda^\vee)^{I_\fq}\right)=0$$
\end{lem}

\begin{proof}
We consider the case $\Lambda=\Lambda_K$; the other cases are similar.

     Let $\alpha_\fq:G_{K_\fq}\ra \BZ_p^\times$ be the unramified character associated to $T_p\tilde{E}_{/\BF_\fq}$. The action of $G_{K_\fq}$ on the module $T_p\tilde{E}_{/\BF_\fq}\otimes\Lambda_K^\vee$ is then given by the character $\alpha_\fq\varepsilon_K^{-1}$.
     Let $C_\fq\subset\Lambda_K$ be the ideal generated by the set $\{\varepsilon_K(g)-1:g\in I_\fq\}$, and let $\sigma_\fq\in G_{K_\fq}$ be a lift of $\Frob_\fq$. A standard calculation yields the isomorphism:
     \[H^1\left(G_{K_\fq}/I_\fq,(T_p\tilde{E}_{/\BF_\fq}\otimes\Lambda_K^\vee)^{I_\fq}\right)\simeq \Hom_{\mathrm{cont}}\left((\Lambda_K/C_\fq)^{\varepsilon_K(\sigma_\fq)=\alpha_\fq^{-1}(\Frob_\fq)},\BQ_p/\BZ_p\right),\]
     where $(\Lambda_K/C_\fq)^{\varepsilon_K(\sigma_\fq)=\alpha_\fq^{-1}(\Frob_\fq)}$ denotes the submodule of $\Lambda_K/C_\fq$ annihilated by the operator $\varepsilon_K(\sigma_\fq)-\alpha_\fq^{-1}(\Frob_\fq)$.
    By \cite[Lemma 12.5.4]{Nek}, for any embedding $\sigma:\ov{\BQ}\ra \BC$, $|\sigma(\alpha_\fq(\Frob_\fq))|=p^{1/2}$. Thus, $\alpha_\fq(\Frob_\fq)$ is not a root of unity. Consequently, the submodule $(\Lambda_K/C_\fq)^{\varepsilon_K(\sigma_\fq)=\alpha_\fq^{-1}(\Frob_\fq)}$ must be trivial, which completes the proof.
\end{proof}
\begin{remark}\label{coin}
    Lemma \ref{pcond} shows that the ordinary local condition $H^1_{\ord}$ defined here coincides with that in  \cite{skinner2014iwasawa}.
\end{remark}

\begin{lem}\label{sum}
    We have the following isomorphisms of $\Lambda_K^+$-modules:
    \[ \mathcal{X}_{\mathcal{F}_{\ord}}(E/{K}_\infty^+)\simeq\mathcal{X}_{\mathcal{F}_{\ord}}(E/{\BQ}_\infty)\oplus \mathcal{X}_{\mathcal{F}_{\ord}}(E^K/{\BQ}_\infty) \]
    and
    \[ \mathcal{X}_{\mathcal{F}_{\ord}}^{\Sigma}(E/{K}_\infty^+)\simeq\mathcal{X}_{\mathcal{F}_{\ord}}^{\Sigma}(E/{\BQ}_\infty)\oplus \mathcal{X}_{\mathcal{F}_{\ord}}^{\Sigma}(E^K/{\BQ}_\infty), \]
    via the natural identification $\Lambda_K^+\simeq \Lambda_\BQ$, where \(E^K\) denotes the quadratic twist of \(E\) associated to the quadratic extension \(K/\BQ\).
\end{lem}
\begin{proof}
    See \cite[Lemma 3.6]{skinner2014iwasawa}.
\end{proof}

For any prime $\mathfrak{q}$ of $K$, let $I_\mathfrak{q}$ denote the inertia subgroup of $G_{K_\mathfrak{q}}$.
\begin{lem}\label{loc}
    Let $\Lambda$ be $\Lambda_K$ or $\Lambda_K^+$.
    For any prime $\mathfrak{q} \nmid p$ of $K$, we have:
    \begin{enumerate}
        \item $H^1(G_{K_\mathfrak{q}}/I_\mathfrak{q},(T_pE\otimes\Lambda^\vee)^{I_\mathfrak{q}})=0$.
        \item $\operatorname{Char}_{\Lambda}((H^1(I_\mathfrak{q},T_pE\otimes\Lambda^\vee)^{G_{K_\mathfrak{q}}})^\vee)=(P_\mathfrak{q}(\varepsilon_K(\operatorname{Frob}_\mathfrak{q}^{-1})))$, where
        $$P_\mathfrak{q}(X)=\det\left(1-\operatorname{Nm}(\mathfrak{q})^{-1}X\cdot\operatorname{Frob}_\mathfrak{q}\mid V_pE^{I_\mathfrak{q}}\right).$$
    \end{enumerate}
\end{lem}
\begin{proof}
    We consider the case $\Lambda=\Lambda_K$; the other case is similar.
    For simplicity, let $M:=T_pE\otimes\Lambda_{K}^\vee$ and $N:=T_pE\otimes\Lambda_{K}$. We have a natural $G_K$-equivariant perfect pairing
    $$M\times N\rightarrow \mathbb{Q}_p/\mathbb{Z}_p(1),$$
    which induces the perfect local Tate pairing
    $$H^1(K_\mathfrak{q},M)\times H^1(K_\mathfrak{q},N)\rightarrow \mathbb{Q}_p/\mathbb{Z}_p$$
    for every place $\mathfrak{q}$. Let $\mathfrak{q}\nmid p$ be a prime of $K$.

    (1) Since $G_{K_\mathfrak{q}}/I_\mathfrak{q}=\langle\operatorname{Frob}_\mathfrak{q}\rangle\simeq\widehat{\mathbb{Z}}$ has cohomological dimension $1$, the groups $H^1(G_{K_\mathfrak{q}}/I_\mathfrak{q},N^{I_\mathfrak{q}})$ and $H^1(G_{K_\mathfrak{q}}/I_\mathfrak{q},M^{I_\mathfrak{q}})$ annihilate each other under local Tate duality. The inflation-restriction exact sequence is
    \[0\to H^1(G_{K_\mathfrak{q}}/I_\mathfrak{q},N^{I_\mathfrak{q}})\to H^1(G_{K_\mathfrak{q}},N)\to H^1(I_\mathfrak{q},N)^{G_{K_\mathfrak{q}}}\to 0.\]
    Thus, to prove (1), it suffices to show that $H^1(I_\mathfrak{q},N)^{G_{K_\mathfrak{q}}}=0$.
    
    Since $\mathfrak{q}\nmid p$, the inertia group $I_\mathfrak{q}$ contains a unique closed subgroup $J_\mathfrak{q}$ such that $I_\mathfrak{q}/J_\mathfrak{q}\simeq\mathbb{Z}_p(1)$ as a $\operatorname{Gal}(K_\mathfrak{q}^{\mathrm{ur}}/K_\mathfrak{q})$-module, and $J_\mathfrak{q}$ has pro-finite order prime to $p$. The inflation-restriction sequence yields
    $$H^1(I_\mathfrak{q},N) \cong H^1(I_\mathfrak{q}/J_\mathfrak{q},N^{J_\mathfrak{q}}).$$
    Let $N_0:=T_pE^{J_\mathfrak{q}}$. Since $\varepsilon_K(I_\mathfrak{q})=1$, we have $N^{J_\mathfrak{q}}=N_0\otimes\Lambda_{K}$. Let $\sigma_\mathfrak{q}$ be a topological generator of $I_\mathfrak{q}/J_\mathfrak{q}$. For any cocycle $\phi$ representing a class in $H^1(I_\mathfrak{q}/J_\mathfrak{q},N_0\otimes\Lambda_{K})^{G_{K_\mathfrak{q}}}$ and any $g\in G_{K_\mathfrak{q}}$, we have the relation $g\cdot\phi=\phi$. This implies that for some integer $r=r(g)$,
    $$g^{-1}(\phi(\sigma_\mathfrak{q}))=\phi(g^{-1}\sigma_\mathfrak{q} g)=\phi(\sigma_\mathfrak{q}^r)=(1+\sigma_\mathfrak{q}+\cdots+\sigma_\mathfrak{q}^{r-1})\phi(\sigma_\mathfrak{q}).$$ 
    As $\mathfrak{q}$ does not split completely in $K_\infty/K$, the image $\varepsilon_K(G_{K_\mathfrak{q}})$ is nontrivial. We can therefore choose $g$ such that $\varepsilon_K(g)\neq 1$. However, since $\varepsilon_K(\sigma_\mathfrak{q})=1$, it follows that $\phi(\sigma_\fq)=0$, which implies that $H^1(I_\mathfrak{q},N)^{G_{K_\mathfrak{q}}}=0$.

    (2) Now we prove the second part. As before, we have the duality
    \[(H^1(I_\mathfrak{q},M)^{G_{K_\mathfrak{q}}})^\vee\simeq H^1(G_{K_\mathfrak{q}}/I_\mathfrak{q},N^{I_\mathfrak{q}}).\]
    Since
    $$H^1(G_{K_\mathfrak{q}}/I_\mathfrak{q},N^{I_\mathfrak{q}})\simeq(T_pE^{I_\mathfrak{q}}\otimes\Lambda_{K})/(1-\operatorname{Frob}_\mathfrak{q}^{-1}),$$
    it follows that
    $$\mathrm{Fitt}_{\Lambda_{K}}((H^1(I_\mathfrak{q},M)^{G_{K_\mathfrak{q}}})^\vee)=(P_\mathfrak{q}(\varepsilon_K(\operatorname{Frob}_\mathfrak{q}^{-1}))).$$
    This Fitting ideal is principal and therefore divisorial. The characteristic ideal $\operatorname{Char}_{\Lambda_{K}}((H^1(I_\mathfrak{q},M)^{G_{K_\mathfrak{q}}})^\vee)$ is the minimal divisorial ideal containing the Fitting ideal. Hence, we have
    $$\operatorname{Char}_{\Lambda_{K}}((H^1(I_\mathfrak{q},M)^{G_{K_\mathfrak{q}}})^\vee)=(P_\mathfrak{q}(\varepsilon_K(\operatorname{Frob}_\mathfrak{q}^{-1}))).$$
\end{proof}

\begin{lem}\label{bdp_des}
    The natural $\Lambda_K^-$-module homomorphism
    \[\mathcal{X}_{\mathcal{F}_\Gr}(E/K_\infty)/(\gamma^+-1)\mathcal{X}_{\mathcal{F}_\Gr}(E/K_\infty)\rightarrow\mathcal{X}_{\mathcal{F}_\Gr}(E/K_\infty^-)\]
    is a pseudo-isomorphism.
\end{lem}

\begin{proof}
    The result follows from \cite[Proposition 3.1]{Wan} combined with Lemma \ref{loc}(1). Note that in our case, every finite place in $\Sigma$ splits in $K$.
\end{proof}

We also consider an imprimitive Selmer group. For $M=T_p E\otimes \Lambda_K^\vee$, we define
\[H^1_{\mathcal{F}^\Sigma_{\ord}}(K,M) := \ker\left(H^1(G_{K,\Sigma},M)\rightarrow \prod_{\mathfrak{q}|p}\frac{H^1(K_{\mathfrak{q}},M)}{H^1_{\ord}(K_{\mathfrak{q}},M)}\right)\]
and its Pontryagin dual $\mathcal{X}^{\Sigma}_{\mathcal{F}_\ord}(E/K_\infty) := H^1_{\mathcal{F}^\Sigma_{\ord}}(K,M)^\vee$.

\begin{lem}\label{trivial}
    Suppose that the residual representation $\ov{\rho}_E|_{G_K}:G_K\ra \Aut(E[p])$ is irreducible. Then there is no nontrivial
$\BF_p$-subquotient of $E[p]$ on which $G_{K_\infty}$ acts trivially.
\end{lem}
\begin{proof}
     Let $L$ be the fixed field of $\ker(\ov{\rho}_E|_{G_K})$. By the properties of the Weil pairing, we have $L\supset K(\mu_p)$.
 Let $L_0:=L\cap K_\infty$. Then $\Gal(L/L_0)$ is the image of the natural map $G_{K_\infty}\ra \Gal(L/K)$.
 
 We first claim that if the restriction $\ov{\rho}_E|_{G_{L_0}}$ is reducible, then $p\nmid \#(\Gal(L/L_0))$. Indeed,  since $G_{L_0}$ is a normal subgroup of $G_K$, if $W\subset E[p]$ is a one-dimensional $\BF_p$-subspace invariant under $G_{L_0}$, then for any $g\in G_K$, the subspace $g W\subset E[p]$ is also invariant under $G_{L_0}$. By the irreducibility of $\ov{\rho}_E|_{G_K}$, there exists a $g_0\in G_K$ such that $g_0 W \neq W$. Consequently, $E[p]=W\oplus g_0 W$ as a $G_{L_0}$-module. We can therefore choose an $\BF_p$-basis of $E[p]$ such that the image of $G_{L_0}$ is contained in the diagonal matrix group. Since the order of the diagonal matrix group is $(p-1)^2$, which is coprime to $p$, we conclude that $p \nmid \#\Gal(L/L_0)$.

Now, suppose there exists a one-dimensional $G_{K_\infty}$-stable $\BF_p$-subspace $C \subset E[p]$ such that the action on $C$ or on $E[p]/C$ is trivial; in particular, the restriction $\ov{\rho}_E|_{G_{K_\infty}}$ is reducible. Let $P \in C-\{0\}$, and choose $Q\in E[p]-C$. Since the action of $G_{K_\infty}$ factors through $\Gal(L/L_0)$, the restriction $\ov{\rho}_E|_{G_{L_0}}$ is reducible, and for any $\sigma\in \Gal(L/L_0)$, we have
\[\sigma P=\alpha(\sigma) P, \quad \sigma Q=\beta(\sigma) P+\gamma(\sigma) Q\]
with coefficients $\alpha(\sigma), \beta(\sigma), \gamma(\sigma)\in \BF_p$ satisfying
\begin{enumerate}
    \item either $\alpha(\sigma) = 1$ for all $\sigma \in \Gal(L/L_0)$, or $\gamma(\sigma) = 1$ for all $\sigma \in \Gal(L/L_0)$.
    \item $\alpha(\sigma)\gamma(\sigma)=\omega(\sigma)$, where $\omega$ is the Teichm\"{u}ller character.
\end{enumerate}
Consider the subgroup $\Gal(L/L_0(\mu_p))\subset \Gal(L/L_0)$. For any $\sigma\in \Gal(L/L_0(\mu_p))$, we have $\omega(\sigma)=1$, hence $\alpha(\sigma)=\gamma(\sigma)=1$. Therefore we obtain an embedding 
    \[\Gal(L/L_0(\mu_p)) \hookrightarrow \left\{\begin{pmatrix}
        1 & a\\
        0 & 1
    \end{pmatrix}\,\middle|\, a\in \BF_p\right\} \simeq \BF_p.\]
Since $\Gal(L/L_0(\mu_p))$ is a subgroup of $\Gal(L/L_0)$ and $p\nmid \#(\Gal(L/L_0))$ as shown previously, we have that $L=L_0(\mu_p)$. 
On the other hand, since $[K(\mu_p):K]$ divides $p-1$ and $[L_0:K]$ is a power of $p$, we have that
$L_0\cap K(\mu_p)=K$. Hence $L$ is an abelian extension of $K$, with Galois group $\Gal(L/K)\simeq \Gal(L/L_0)\times \Gal(L/K(\mu_p))$, and therefore  the representation $\ov{\rho}_E|_{G_{K}}$ is reducible over $\ov{\BF}_p$.  Moreover, let $\alpha(\sigma), \gamma(\sigma)\in \ov{\BF}_p^\times$ denote the eigenvalues of $\sigma \in \Gal(L_0(\mu_p)/K)$. We have 
\begin{enumerate}
    \item If $\sigma\in \Gal(L/L_0)$, then at least one of $\alpha(\sigma), \gamma(\sigma)$ is equal to one, and $\alpha(\sigma)\cdot\gamma(\sigma)=\omega(\sigma)\in \BF_p^\times$, hence both $\alpha(\sigma)$ and $\gamma(\sigma)$ belong to $\BF_p^\times$.
    \item If $\sigma\in \Gal(L/K(\mu_p))\simeq \Gal(L_0/K)$, since $\Gal(L_0/K)$ has $p$-power order, $\alpha(\sigma)^{p^m}=1=\gamma(\sigma)^{p^m}\in \ov{\BF}_p^\times$ for some integer $m$, which implies that $\alpha(\sigma)=1=\gamma(\sigma)\in {\BF}_p^\times$.
\end{enumerate}
Consequently, for any $\sigma\in \Gal(L/K)$, the eigenvalues $\alpha(\sigma), \gamma(\sigma)$ belong to $\BF_p^\times$, contradicting the assumption that $\ov{\rho}_E|_{G_K}$ is irreducible. We complete the proof.
\end{proof}

\begin{lem}
    If the residual representation $\ov{\rho}_E|_{G_K}$ is irreducible, then 
    $\mathcal{X}^{\Sigma}_{\mathcal{F}_\ord}(E/K_\infty)$ is a torsion $\Lambda_K$-module.
\end{lem}
\begin{proof}
    By \cite[Theorem 17.4 (1)]{Kato}, both $\mathcal{X}_{\mathcal{F}_{\ord}}^{\Sigma}(E/\mathbb{Q}_\infty)$ and $\mathcal{X}_{\mathcal{F}_{\ord}}^{\Sigma}(E^{K}/\mathbb{Q}_\infty)$ are  torsion $\Lambda_\BQ$-modules. Consequently, Lemma \ref{sum} implies that $\mathcal{X}_{\mathcal{F}_{\ord}}^{\Sigma}(E/{K}_\infty^+)$ is  a torsion $\Lambda_K^+$-module. Therefore, by \cite[Proposition 3.9]{skinner2014iwasawa}, Lemma \ref{trivial}, and Remark \ref{coin}, we conclude that $\mathcal{X}^{\Sigma}_{\mathcal{F}_\ord}(E/K_\infty)$ is a torsion $\Lambda_K$-module.
\end{proof}

Together with Lemma \ref{loc}, we have the following corollary.

\begin{cor}\label{impr}
    $\CX_{\CF_\ord}(E/K_\infty)$ is $\Lambda_K$-torsion, and
    \[\Char_{\Lambda_K}(\CX^{\Sigma}_{\CF_\ord}(E/K_\infty))\supset\Char_{\Lambda_K}(\CX_{\CF_\ord}(E/K_\infty))\prod_{\fq\in\Sigma-\{\fp,\bar{\fp}\}}(P_\fq(\varepsilon_K(\Frob_\fq^{-1}))).\]
\end{cor}

\begin{proof}
    By definition, we have the following exact sequence,
    \[0\ra H^1_{\mathcal{F}_{\ord}}(K,T_pE\otimes\Lambda_K^\vee)\ra H^1_{\mathcal{F}_{\ord}^{\Sigma}}(K,T_pE\otimes\Lambda_K^\vee)\ra \bigoplus_{\fq\in\Sigma-\{\fp,\bar{\fp}\}}H^1(G_{K_\mathfrak{q}},T_pE\otimes\Lambda_K^\vee).\]
  Taking Pontryagin duals yields the exact sequence,
    \[ \bigoplus_{\fq\in\Sigma-\{\fp,\bar{\fp}\}}H^1(G_{K_\mathfrak{q}},T_pE\otimes\Lambda_K^\vee)^\vee\ra \CX_{\CF_\ord^\Sigma}(E/K_\infty)\ra \CX_{\CF_\ord}(E/K_\infty) \ra 0.\]
  Hence we have
    \[\Char_{\Lambda_K}(\CX^{\Sigma}_{\CF_\ord}(E/K_\infty))\supset\Char_{\Lambda_K}(\CX_{\CF_\ord}(E/K_\infty))\prod_{\fq\in\Sigma-\{\fp,\bar{\fp}\}}\Char_{\Lambda_K}(H^1(G_{K_\mathfrak{q}},T_pE\otimes\Lambda_K^\vee)^\vee).\]
    For every $\fq\nmid p$, by the inflation-restriction sequence and the fact that $G_{K_\mathfrak{q}}/I_\mathfrak{q}=\langle\operatorname{Frob}_\mathfrak{q}\rangle\simeq\widehat{\mathbb{Z}}$ has cohomological dimension $1$, we have the following exact sequence
    \[0\to H^1(G_{K_\mathfrak{q}}/I_\mathfrak{q},(T_pE\otimes\Lambda_K^\vee)^{I_\mathfrak{q}})\to H^1(G_{K_\mathfrak{q}},T_pE\otimes\Lambda_K^\vee)\to H^1(I_\mathfrak{q},T_pE\otimes\Lambda_K^\vee)^{G_{K_\mathfrak{q}}}\to 0.\]
    Therefore by Lemma \ref{loc}, we have
    \[\Char_{\Lambda_K}(H^1(G_{K_\mathfrak{q}},(T_pE\otimes\Lambda_K^\vee))^\vee)=\Char_{\Lambda_K}((H^1(I_\mathfrak{q},T_pE\otimes\Lambda_K^\vee)^{G_{K_\mathfrak{q}}})^\vee)=P_\fq(\varepsilon_K(\Frob_\fq^{-1})),\]
    which completes the proof.
\end{proof}

\subsection{Selmer groups for Hida families}\label{Hida}
Let $W$ be an indeterminate and $\Lambda_W=\mathbb{Z}_p[[W]]$. We identify $\Lambda_W$ with $\Lambda_K^+$ by mapping the topological generator $\gamma^+$ to $1+W$, which induces the character $\varepsilon_W:G_K\rightarrow \Lambda_W^\times$.

Let $L\subset\bar{\BQ}_p$ be a finite extension of $\BQ_p$ with ring of integers $\CO$. Let $\BI$ be a local reduced finite integral extension of $\Lambda_{W,\CO}:=\Lambda_W\otimes_{\BZ_p}\CO$.
Recall that a continuous $\CO$-algebra homomorphism $\phi\in \Hom_{\text{cont }\CO\text{-alg}}(\BI,\bar{\BQ}_p)$ is said to be arithmetic if it satisfies the condition $\phi(1+W)=\zeta_\phi(1+p)^{k_\phi-2}$ for some $p$-power root of unity $\zeta_\phi$ and some integer $k_\phi$. Let $t_\phi>0$ be the integer such that $\zeta_\phi$ is a primitive $p^{t_\phi-1}$-th root of unity, and let $\chi_\phi:\BA_{\BQ}^\times/\BQ^{\times}\rightarrow\mu_{p^\infty}$ be the unique character that has a $p$-power conductor and satisfies $\chi_{\phi,p}(1+p)=\zeta_\phi^{-1}$.
We define the set of arithmetic points with weight $k_\phi \ge 2$ as
$$\fX_{\BI,\CO}^a =\left\{ \phi\in \Hom_{\text{cont }\CO\text{-alg}}(\BI,\bar{\BQ}_p): \phi \text{ is arithmetic }, k_\phi\geq 2\right\}.$$

Let $\mathbf{f}$ be an $\BI$-adic ordinary eigenform of tame level $N$ and nebentypus ${\chi}_{\mathbf{f}} = 1$, namely a formal $q$-expansion $\mathbf{f}=\sum_{n=1}\mathbf{a}_n(\mathbf{f})q^n\in\BI[[q]]$ with the property that for any arithmetic point $\phi\in\fX_{\BI,\CO}^a$, the corresponding specialization
$$\mathbf{f}_{\phi}=\sum_{n=1}\phi(\mathbf{a}_n(\mathbf{f}))q^n$$
is an ordinary modular form in the space $M_{k_\phi}(Np^{t_\phi},\omega^{k_\phi-2}\chi_\phi;\phi(\BI))$. Here, $\omega$ denotes the Teichmüller character.

We assume that $\mathbf{f}$ satisfies the following irreducibility condition:
\begin{equation}\tag{$\irred_\mathbf{f}$}
    \text{the residue representation of }\rho_{\mathbf{f}_\phi} \text{ is irreducible for some (hence all) }\phi\in \fX_{\BI,\CO}^a.
\end{equation}
Under this condition, there exists a continuous, $\BI$-linear Galois representation $\rho_{\mathbf{f}}:G_\BQ\rightarrow \GL_{\BI}(T_{\mathbf{f}})$, where the representation space $T_{\mathbf{f}}$ is a free $\BI$-module of rank two. This representation is unramified at all primes $\ell\nmid Np$ and is uniquely characterized by the relations:
$$\tr\rho_{\mathbf{f}}(\Frob_\ell)=a_\ell{(\mathbf{f})},\quad \ell\nmid Np,$$
and
$$\det(\rho_{\mathbf{f}})=\epsilon\varepsilon_W,$$
where $\epsilon$ is the $p$-adic cyclotomic character.

Let $F_\BI$ be the field of fractions of the integral domain $\BI$, and let $V_\mathbf{f}:=T_\mathbf{f}\otimes_\BI F_\BI$. Since $\mathbf{f}$ is ordinary,  there exists a one-dimensional $F_\BI$-subspace $V_\mathbf{f}^+\subset V_\mathbf{f}$ that is stable under the action of $G_{\BQ_p}$. Furthermore, $G_{\BQ_p}$ acts on the quotient space $V_\mathbf{f}^-:=V_\mathbf{f}/V_\mathbf{f}^+$ via an unramified character $\delta_{\mathbf{f}}$, which is characterized by the condition $\delta_{\mathbf{f}}(\Frob_p)=a_p(\mathbf{f})$. Since the nebentypus $\chi_{\mathbf{f}}=1$, $\mathbf{f}$ satisfies the following condition given in \cite{skinner2014iwasawa}:
\begin{equation}\tag{$\text{dist}_\mathbf{f}$}
    \delta_{\mathbf{f}}^2\not\equiv\chi_{\mathbf{f}}\omega \mod \fm_\BI,
\end{equation}
where $\fm_\BI$ is the maximal ideal of $\BI$. Hence, by Nakayama's lemma, the intersection $T_\mathbf{f}^+:=T_\mathbf{f}\cap V_\mathbf{f}^+$ is then a rank-one free $\BI$-summand of $T_\mathbf{f}$. Consequently, the quotient module $T_\mathbf{f}^-:=T_\mathbf{f}/T_\mathbf{f}^+$ is also a free $\BI$-module of rank one.

Let $E/\BQ$ be an elliptic curve with good ordinary reduction at $p$. By \cite[Corollary 1.5]{Hida86}, there exist  an $\BI$-adic ordinary eigenform  $\mathbf{f}$ with trivial nebentypus and a weight 2 arithmetic specialization $\phi\in\fX_{\BI,\CO}^a$  (i.e., $\phi(1+W)=1$) such that   the specialized form $\phi(\mathbf{f})$ is the $p$-stabilization of the newform associated to $E$.

\begin{cor}\label{Hida_des}
    Let $\fp_\phi:=\ker\phi$. Assume that the residual representation $\ov{\rho}_E|_{G_K}:G_K\ra \Aut(E[p])$ is irreducible. Then there is a natural isomorphism:
    $$\CX^{\Sigma}_{\CF_\ord}(\mathbf{f}/K_\infty)/\fp_\phi \CX^{\Sigma}_{\CF_\ord}(\mathbf{f}/K_\infty)\otimes_{\BI,\phi}\CO\simeq \CX^{\Sigma}_{\CF_\ord}(E/K_\infty).$$
\end{cor}

\begin{proof}

     Since $\mathbf{f}$ is an $\BI$-adic ordinary eigenform of trivial nebentypus, recall that there exists a free $\BI$-submodule $T_\mathbf{f}^+\subset T_\mathbf{f}$ of rank one such that the quotient module $T_\mathbf{f}^-:=T_\mathbf{f}/T_\mathbf{f}^+$ is also free of rank one, and the inertia group $I_p$ acts trivially on $T_\mathbf{f}^-$. By Lemma \ref{trivial} and Nakayama's lemma, there is no nontrivial $\BI_K$-subquotient of $T_\mathbf{f}^\vee$ on which $G_{K_\infty}$ acts trivially. 

     Consequently, the desired result follows from \cite[Proposition 3.7]{skinner2014iwasawa} and Remark \ref{coin}. 
   
 \end{proof}

\section{\texorpdfstring{$p$-adic $L$-functions}{p-adic L-functions}}

\subsection{\texorpdfstring{A three-variable \(p\)-adic \(L\)-function}{A three-variable p-adic L-function}}
Let \(\mathbf{f}\) be an \(\BI\)-adic ordinary eigenform of tame level \(N\) and trivial nebentypus as described in Section \ref{Hida}. Suppose \(L\) contains \(\BQ[\mu_{Np}, i, D_K^{1/2}]\).

Define
\[
\fX^a_{\BI_K,\CO}:=\{\phi\in\Hom_{\text{cont }\CO\text{-alg}}(\BI_K,\bar{\BQ}_p): \phi|_{\BI}\in\fX^a_{\BI,\CO}, \phi(\gamma^+)=\zeta_+(1+p)^{k_{\phi|_{\BI}}-2}, \phi(\gamma^-)=\zeta_-\},
\]
where \(\zeta_\pm\) are \(p\)-power roots of unity. For each \(\phi\in\fX^a_{\BI_K,\CO}\), let \(k_\phi\), \(t_\phi\), and \(\chi_\phi\) be the corresponding data associated with \(\phi|_{\BI}\). Define
\[
\xi_\phi:=\phi\circ(\varepsilon_K/\varepsilon_W),\quad \theta_\phi:=\omega^{2-k_\phi}\chi_\phi^{-1}\xi_\phi.
\]
These are finite-order idele class characters of \(\BA_K^\times\). For an idele class character \(\psi\), denote its conductor by \(\ff_\psi\). Set
\[
\fX'_{\BI_K,\CO}:=\{\phi\in\fX^a_{\BI_K,\CO}: p\mid\ff_{\xi_\phi}, p^{t_\phi}\mid\Nm(\ff_{\xi_\phi}), p\mid\ff_{\theta_\phi}\}.
\]

\begin{thm}
Let \(\Sigma\) be a finite set of primes containing all primes dividing \(pND_K\). Let \(\mathbf{f}\) be an \(\BI\)-adic ordinary eigenform of tame level \(N\) and trivial nebentypus. Assume \(\mathbf{f}\) satisfies \((\irred_{\mathbf{f}})\). Then there exists \(\CL^\Sigma_{\mathbf{f},K}\in\BI_K\) such that for every \(\phi\in\fX'_{\BI_K,\CO}\),
\begin{align*}
\CL^\Sigma_{\mathbf{f},K}(\phi)=&u_{\mathbf{f}_\phi}a_p(\mathbf{f}_\phi)^{-\ord_p(\Nm(\ff_{\theta_\phi}))}\\
&\quad\times\frac{((k_\phi-2)!)^2\fg(\theta_\phi^{-1})\Nm(\ff_{\theta_\phi}\delta_K)^{k_\phi-2}L^\Sigma(\mathbf{f}_\phi/K,{\theta_\phi^{-1}},{\theta_\phi}-1)}{(-2\pi i)^{2k_\phi-2}\Omega^+_{\mathbf{f}_\phi}\Omega^-_{\mathbf{f}_\phi}},
\end{align*}
where \(u_{\mathbf{f}_\phi}\) is a \(p\)-adic unit depending only on \(\mathbf{f}_\phi\), \(\fg(\theta_\phi^{-1})\) is the global Gauss sum, \(\delta_K\) is the differential ideal, \(\Omega^\pm_{\mathbf{f}_\phi}\) are the canonical periods associated to \(\mathbf{f}_\phi\), and 
\[L^\Sigma(\mathbf{f}_\phi/K,{\theta_\phi^{-1}},s):=L^\Sigma(\mathbf{f}_\phi/K,{\theta_\phi^{-1}},s)\cdot\prod_{\fq \in \Sigma-\{\infty\}}L_{\fq}(\mathbf{f}_\phi/K,{\theta_\phi^{-1}},s)^{-1}\]
with 
$$L_{\fq}(\mathbf{f}_\phi/K,{\theta_\phi^{-1}},s):=\det\left(1-\operatorname{Nm}(\mathfrak{q})^{-s}\cdot\operatorname{Frob}_\mathfrak{q}\Bigm| \left(V_{\phi,\ell}|_{G_K}\otimes\theta_\phi^{-1}\right)^{I_\mathfrak{q}}\right)^{-1}$$
being the local Euler factor at $\fq$,
here $\ell\notin \Sigma$ is a prime number and $V_{\phi,\ell}$ is the $\ell$-adic Galois representation associated to $\mathbf{f}_\phi$ (in the manner as Section \ref{Hida}).
\end{thm}
\begin{proof}
    See \cite[Section 3.4.5]{skinner2014iwasawa}.
\end{proof}

\begin{remark}
Note that we adopt the arithmetic Frobenius normalization for the reciprocity map of class field theory, whereas \cite{skinner2014iwasawa} employs the geometric Frobenius convention.
\end{remark}

\subsection{\texorpdfstring{Two variable $p$-adic $L$-functions: type I} {Two variable p-adic L-functions: type I}}

Let \(f=\sum_n a_n q^n \in S_2(\Gamma_0(N))\) be a newform with \(p\nmid a_p\), and let \(c_f \in \BZ_p\) be the congruence number associated with \(f\) as defined in \cite[Section 7]{hida81} or \cite{rib83}.

\begin{thm}
There exists an element \(\CL_p^I(f/K)\in c_f^{-1}\Lambda_K\) such that for every finite-order nontrivial character \(\xi\) of \(\Gamma_K\) of conductor $\fp^m\ov{\fp}^n$ with $m + n > 0$, we have
\[
\CL_p^I(f/K)(\xi)=W(\xi)p^{\ord_p(\Nm(\ff_\xi))/2}\alpha_p^{-\ord_p(\Nm(\ff_\xi))}\left(1-\frac{p}{\alpha_p^2}\right)^{-1}\left(1-\frac{1}{\alpha_p^2}\right)^{-1}\frac{L(f/K,\xi^{-1},1)}{8\pi^2\langle f,f\rangle},
\]
where \(\alpha_p\) is the \(p\)-adic unit root of the polynomial \(x^2 - a_p x + p\), \(\langle f,g\rangle=\int_{\Gamma_0(N)\backslash\CH}\overline{f(\tau)}g(\tau)d\tau\) denotes the Petersson inner product on \(S_2(\Gamma_1(N))\), and \(W(\xi)\) is the Artin root number.
\end{thm}
\begin{proof}
    This is a reformulation of \cite[Theorem 1.1]{PR86} following  Hida's $p$-adic Rankin method \cite{Hida85},  as given in \cite[Theorem 2.2.1]{CGS}.
\end{proof}

Let \(E/\BQ\) be an elliptic curve with good ordinary reduction at the prime \(p\). Let \(f_E\in S_2(\Gamma_0(N))\) denote the newform associated with \(E\), and let \(\pi_E: X_0(N)\rightarrow E\) be a modular parametrization.

\begin{defn}[Perrin-Riou's \(p\)-adic \(L\)-function]
Define Perrin-Riou's \(p\)-adic \(L\)-function as
\[
\CL_p^{\mathrm{PR}}(E/K):=\left(1-\frac{p}{\alpha_p^2}\right)\left(1-\frac{1}{\alpha_p^2}\right)\cdot\frac{\deg(\pi_E)}{c_E^2}\cdot \CL_p^I(f_E/K)\in \Lambda_K,
\]
where \(c_E\) denotes the Manin constant.
\end{defn}

Let \(\mathbf{f}\) be an \(\BI\)-adic ordinary eigenform and \(\phi\in\fX_{\BI,\CO}^a\) as in Section \ref{Hida}, with \(\phi(1+W)=1\) and \(\mathbf{f}_\phi\) equal to the \(p\)-stabilization of \(f_E\).

\begin{prop}\label{imp0}
Assume that the residual representation \(\bar{\rho}_E:G_{\BQ}\rightarrow \Aut(E[p])\) is irreducible. Then
\[
\CL^\Sigma_{\mathbf{f},K}(\phi\otimes\id)=\alpha\cdot\prod_{\fq\in\Sigma,\fq\nmid p}P_\fq(\varepsilon_K(\Frob_\fq^{-1}))\cdot\CL_p^{\mathrm{PR}}(E/K),
\]
where
\[
P_\fq(X)=\det(1-\Nm(\fq)^{-1}X\cdot\Frob_\fq\mid V_pE^{I_\fq})
\]
is the local Euler factor, \(\alpha\) is a \(p\)-adic unit independent of \(\xi\), and
\[
\phi\otimes\id:\BI_K=\BI\hat{\otimes}\Lambda_K\rightarrow\phi(\BI)\otimes\Lambda_K
\]
is the natural homomorphism.
\end{prop}

\begin{proof}
We have \(W(\xi)=g(\bar{\xi})/\sqrt{\Nm(\ff_\xi)}\). By \cite[Corollary 4.1]{Mazur}, \(c_E\) is a \(p\)-adic unit. According to \cite[Lemma 4.1.2]{CGS}, \(\deg(\pi_E)=c_{f_E}\) up to a \(p\)-adic unit. Additionally, \cite[Lemma 9.5]{skinner2014indivisibility} shows that
\[
\frac{\langle f_E,f_E\rangle}{c_{f_E}}=i(2\pi i)^2\Omega^+_{f_E}\Omega^-_{f_E}.
\]
These results establish the desired equality.
\end{proof}

\subsection{\texorpdfstring{Cyclotomic $p$-adic $L$-function} {Cyclotomic p-adic L-function}} 

Let \(E/\mathbb{Q}\) be an elliptic curve with good ordinary reduction at the prime \(p\). Choose generators \(\delta^\pm\) of \(H_1(E,\mathbb{Z})^\pm\), and define the \Neron periods \(\Omega_E^\pm\) by
\[
\Omega_E^\pm = \int_{\delta^\pm} \omega_E,
\]
where \(\omega_E\) is a minimal differential on \(E\). We normalize \(\delta^\pm\) so that \(\Omega_E^+ \in \mathbb{R}_{>0}\) and \(\Omega_E^- \in i\mathbb{R}_{>0}\).

Let \(a_p\) be the $p$-th Fourier coefficient of the newform \(f_E\) associated with \(E\), and let \(\alpha_p\) be the \(p\)-adic unit root of \(x^2 - a_p x + p\), as defined previously.

\begin{thm}
There exists an element \(\mathcal{L}_p^{\mathrm{MSD}}(E/\mathbb{Q}) \in \Lambda_{\mathbb{Q}}\) such that for every finite-order character \(\chi\) of \(\Gamma_{\mathbb{Q}}\),
\[
\mathcal{L}_p^{\mathrm{MSD}}(E/\mathbb{Q})(\chi) = \begin{cases}
\frac{p^r}{g(\overline{\chi})\alpha_p^r} \cdot \frac{L(E,\overline{\chi},1)}{\Omega_E^+},\quad & \chi\text{ is of conductor }p^r\neq 1,\\
(1-\alpha_p^{-1})^2 \cdot \frac{L(E,1)}{\Omega_E^+},\quad &\chi=1.
\end{cases}
\]
where \(g(\overline{\chi}) = \sum_{a \bmod p^r} \overline{\chi}(a)e^{2\pi i a/p^r}\) is the Gauss sum.
\end{thm}
\begin{proof}
    See \cite[Theorem 2.1.1]{CGS}.
\end{proof}

Let \(\mathcal{L}_p^{\mathrm{PR}}(E/K)^+ \in \Lambda_{\mathbb{Q}}\simeq \Lambda_K^+\) be the image of \(\mathcal{L}_p^{\mathrm{PR}}(E/K)\) under the map induced by the projection \(\Gamma_K \twoheadrightarrow \Gamma_K^+ \simeq \Gamma_{\mathbb{Q}}\).

\begin{prop}\label{PR-MSD}
We have
\[
\mathcal{L}_p^{\mathrm{PR}}(E/K)^+ = \mathcal{L}_p^{\mathrm{MSD}}(E/\mathbb{Q}) \cdot \mathcal{L}_p^{\mathrm{MSD}}(E^K/\mathbb{Q}),
\]
up to a unit in \(\Lambda_{\mathbb{Q}}^\times\).
\end{prop}
\begin{proof}
    See \cite[Proposition 2.2.4]{CGS}.
\end{proof}
\begin{prop}\label{non1}
    Both $\mathcal{L}_p^{\mathrm{MSD}}(E/\BQ)$ and $\mathcal{L}_p^{\mathrm{PR}}(E/K)$ are non-vanishing.
\end{prop}
\begin{proof}
    According to \cite{rohrlich1984functions}, there exists a finite-order character \(\chi\) of \(\Gamma_{\mathbb{Q}}\) such that $\mathcal{L}_p^{\mathrm{MSD}}(E/\mathbb{Q})(\chi)\neq 0$, which implies $\mathcal{L}_p^{\mathrm{MSD}}(E/\BQ)\neq 0$. Similarly, we have $\mathcal{L}_p^{\mathrm{MSD}}(E^K/\BQ)\neq 0$. By Proposition \ref{PR-MSD}, it follows that $\mathcal{L}_p^{\mathrm{PR}}(E/K)^+\neq 0$, and hence $\mathcal{L}_p^{\mathrm{PR}}(E/K)\neq 0$.
\end{proof}
\subsection{\texorpdfstring{Two variable $p$-adic $L$-functions: type II} {Two variable p-adic L-functions: type II}}

\begin{thm}
There exists an element \(\mathcal{L}_p^{II}(f/K)\in \mathrm{Frac}~\Lambda_K\) such that for every character \(\xi\) of \(\Gamma_K\) crystalline at both \(\mathfrak{p}\) and \(\bar{\mathfrak{p}}\), and of infinity type \((b,a)\) with \(a\leq -1\) and \(b\geq 1\),
\[
\mathcal{L}_p^{II}(f/K)(\xi)=\frac{2^{a-b}i^{b-a-1}\Gamma(b+1)\Gamma(b)N^{a+b+1}}{(2\pi)^{2b+1}\langle\theta_{\xi_b},\theta_{\xi_b}\rangle}\cdot\frac{\mathcal{E}(\xi,f,1)}{(1-\xi^{1-\tau}(\bar{\mathfrak{p}}))(1-p^{-1}\xi^{1-\tau}(\bar{\mathfrak{p}}))}\cdot L(f/K,\xi,1),
\]
where \(\theta_{\xi_b}\) is the theta series associated to the Hecke character \(\xi_b=\xi|\cdot|^{-b}\), \(\tau\) denotes complex conjugation, and
\[
\mathcal{E}(\xi,f,1)=(1-p^{-1}\xi(\bar{\mathfrak{p}})\alpha_p)(1-\xi(\bar{\mathfrak{p}})\alpha_p^{-1})(1-p^{-1}\xi^{-1}(\mathfrak{p})\alpha_p)(1-\xi^{-1}(\mathfrak{p})\alpha_p^{-1}).
\]
\end{thm}
\begin{proof}
    This is another instance of Hida's $p$-adic Rankin L-series  in \cite[Theorem 6.1.3]{LLZ15}, with the roles of $\fp$ and $\ov{\fp}$ reversed. See also \cite[Theorem 2.4.1]{CGS}.
\end{proof}
Let \(\Lambda^{\mathrm{ur}}_K:=\Lambda_K\widehat{\otimes}\mathbb{Z}_p^{\mathrm{ur}}\), where \(\mathbb{Z}_p^{\mathrm{ur}}\) is the completion of the ring of integers of the maximal unramified extension of \(\mathbb{Q}_p\).

\begin{thm}
There exists an element \(\mathcal{L}_{\mathfrak{p}}(K)\in\Lambda^{\mathrm{ur}}_K\) such that for every character \(\xi\) of \(\Gamma_K\) of infinity type \((j,k)\) with \(0\leq -j\leq k\),
\[
\mathcal{L}_{\mathfrak{p}}(K)(\xi)=\frac{\Omega_p^{k-j}}{\Omega_K^{k-j}}\cdot\Gamma(k)\cdot\left(\frac{\sqrt{D_K}}{2\pi}\right)^j\cdot(1-\xi^{-1}(\mathfrak{p})p^{-1})(1-\xi(\bar{\mathfrak{p}}))\cdot L(\xi,0),
\]
where \(\Omega_p\) and \(\Omega_K\) are the CM periods associated with \(K\).
\end{thm}
\begin{proof}
    See \cite[Theorem II.4.14]{deShalit}.
\end{proof}

\begin{defn}[Greenberg's \(p\)-adic \(L\)-function]
Define Greenberg's \(p\)-adic \(L\)-function as
\[
\mathcal{L}_p^{\mathrm{Gr}}(f/K):=h_K\cdot\mathcal{L}_{\mathfrak{p}}(K)'\cdot\mathcal{L}_p^{II}(f/K),
\]
where \(h_K\) denotes the class number of \(K\), and \(\mathcal{L}_{\mathfrak{p}}(K)'\) is the image of \(\mathcal{L}_{\mathfrak{p}}(K)\) under the map \(\Lambda_K^{\mathrm{ur}}\rightarrow \Lambda_K^{\mathrm{ur}}\), \(\gamma\mapsto\gamma^{1-\tau}\) for \(\gamma\in\Gamma_K\).
\end{defn}

\begin{lem}
Greenberg's \(p\)-adic \(L\)-function \(\mathcal{L}_p^{\mathrm{Gr}}(f/K)\) is integral, i.e., it belongs to \(\Lambda_K^{\mathrm{ur}}\).
\end{lem}
\begin{proof}
    See \cite[Lemma 2.4.4]{CGS}.
\end{proof}

\subsection{\texorpdfstring{BDP $p$-adic $L$-function}{BDP p-adic L-function}}

Assume that \(D_K\) is odd, \(D_K \neq -3\), and that the Heegner hypothesis holds. Fix an integral ideal \(\mathfrak{n}\subset\mathcal{O}_K\) satisfying \(\mathcal{O}_K/\mathfrak{n}\simeq \mathbb{Z}/N\mathbb{Z}\). Define \(\Lambda^{\mathrm{ur},\pm}_K:=\Lambda^{\pm}_K\widehat{\otimes}\mathbb{Z}_p^{\mathrm{ur}}\).
\begin{thm}
There exists an element \(\mathcal{L}_p^{\mathrm{BDP}}(f/K)\in \Lambda^{\mathrm{ur},-}_K\) characterized by the following interpolation property: for every character \(\xi\) of \(\Gamma^-_K\) crystalline at both \(\mathfrak{p}\) and \(\bar{\mathfrak{p}}\), corresponding to a Hecke character of \(K\) of infinity type \((n,-n)\) with \(n\in\mathbb{Z}_{>0}\) and \(n \equiv 0 \mod p-1\),
\[
\mathcal{L}_p^{\mathrm{BDP}}(f/K)(\xi)=\frac{\Omega_p^{4n}}{\Omega_K^{4n}}\cdot\frac{\Gamma(n)\Gamma(n+1)\xi(\mathfrak{n}^{-1})}{4(2\pi)^{2n+1}\sqrt{D_K}^{2n-1}}\cdot\left(1-a_p\xi(\bar{\mathfrak{p}})p^{-1}+\xi(\bar{\mathfrak{p}})^2p^{-1}\right)^2\cdot L(f/K, \xi, 1).
\]
\end{thm}
\begin{proof}
    This was originally constructed in \cite{BDP} (as a continuous function of $\xi$).
    See \cite[Theorem 2.1.1]{CGLS} for this refined construction, which is based on \cite[Section 3]{CH}.
\end{proof}

Let \(\mathcal{L}_p^{\mathrm{Gr}}(f/K)^\pm\) denote the image of \(\mathcal{L}_p^{\mathrm{Gr}}(f/K)\) under the natural projection \(\Lambda^{\mathrm{ur}}_K\rightarrow \Lambda^{\mathrm{ur},\pm}_K\). We have the following proposition.

\begin{prop}\label{comp}
\[
\mathcal{L}_p^{\mathrm{Gr}}(f/K)^-\cdot\Lambda^{\mathrm{ur},-}_K=\mathcal{L}_p^{\mathrm{BDP}}(f/K)\cdot \Lambda^{\mathrm{ur},-}_K.
\]
\end{prop}
\begin{proof}
    See \cite[Proposition 2.4.5]{CGS}.
\end{proof}

\section{Main conjectures over $K_\infty$}

\subsection{Two variable Iwasawa main conjectures}

Let $E/\BQ$ be an elliptic curve of conductor $N$ and $p>2$ be a prime at which $E$ has good ordinary reduction. Let $K$ be an imaginary quadratic field in which the prime $p$ splits, i.e., $p\CO_K=\fp\bar{\fp}$. Assume that the residual representation
\[
\bar{\rho}_{E}|_{G_K}: G_K \rightarrow \Aut(E[p])
\]
is irreducible.

Let $f$ be the newform associated to $E$. For the remainder of this paper, we will denote $\CL^\Gr(f/K)$ by $\CL^\Gr(E/K)$, and similarly $\CL_p^{\Gr}(f/K)^\pm$ and $\CL^{\mathrm{BDP}}(f/K)$ by analogous notation involving $E/K$. We consider the following two-variable Iwasawa main conjectures:

\begin{conj}\label{mainconj1}
    \begin{enumerate}
        \item $\CX_{\CF_\ord}(E/K_\infty)$ is $\Lambda_K$-torsion and
        \begin{equation*}
            \Char_{\Lambda_K}(\CX_{\CF_\ord}(E/K_\infty))=(\CL_p^\mathrm{PR}(E/K)).
        \end{equation*}
        \item $\CX_{\CF_\Gr}(E/K_\infty)$ is $\Lambda_K$-torsion and
        \begin{equation*}
            \Char_{\Lambda_K}(\CX_{\CF_\Gr}(E/K_\infty))\Lambda_K^{\ur}=(\CL_p^{\Gr}(E/K)).
        \end{equation*}
    \end{enumerate}
\end{conj}

The main theorem we will establish is the following:

\begin{thm}\label{mainthm1}
Suppose the residual representation $\bar{\rho}_{E}|_{G_K}:G_K\rightarrow \Aut(E[p])$ is absolutely irreducible. If the Heegner hypothesis holds (in particular, $\sign(E/K)=-1$), then:
    \begin{enumerate}
        \item $\CX_{\CF_\ord}(E/K_\infty)$ is $\Lambda_K$-torsion and
        \begin{equation*}
            \Char_{\Lambda_K}(\CX_{\CF_\ord}(E/K_\infty))\subset(\CL_p^\mathrm{PR}(E/K)).
        \end{equation*}
        \item $\CX_{\CF_\Gr}(E/K_\infty)$ is $\Lambda_K$-torsion and
        \begin{equation*}
            \Char_{\Lambda_K}(\CX_{\CF_\Gr}(E/K_\infty))\Lambda_K^{\ur}\subset(\CL_p^{\Gr}(E/K)).
        \end{equation*}
        Moreover, if the following condition 
            \begin{equation}\tag{Im}
        \text{there exists } \tau \in \operatorname{Gal}(\overline{\mathbb{Q}}/\mathbb{Q}(\mu_{p^\infty})) \text{ such that } T_pE / (\rho_E(\tau) - 1)T_pE \text{ is free of rank one over }\BZ_p 
    \end{equation}
        holds, then Conjecture \ref{mainconj1} is true.
    \end{enumerate}
\end{thm}

We will prove Theorem \ref{mainthm1} in the following two subsections.

\subsection{A three variable Iwasawa main conjecture}

Let $\mathbf{f}$ be an $\BI$-adic ordinary eigenform of tame level $N$ and trivial nebentypus as in Section \ref{Hida}. Suppose that $L\supset\BQ[\mu_{Np}, i, D_K^{1/2}]$, $\BI$ is a normal domain, and $\mathbf{f}$ satisfies ($\irred_{\mathbf{f}}$). Note that in our case, $\mathbf{f}$ satisfies ($\text{dist}_{\mathbf{f}}$) since $\mathbf{f}$ has trivial nebentypus.

\begin{conj}[\cite{skinner2014iwasawa}]\label{Hidaconj}
Let $\Sigma$ be a finite set of primes containing all primes dividing $pND_K$. Then
\[
\Char_{\BI_K}(\CX^{\Sigma}_{\CF_\ord}(\mathbf{f}/K_\infty))=(\CL^\Sigma_{\mathbf{f},K}).
\]
\end{conj}

\begin{thm}\label{SU1}
Under the conditions of Conjecture \ref{Hidaconj}, we have, for any height one prime $P$ of $\BI_K=\BI[[\Gamma_K]]$,
\[
\ord_P\left(\Char_{\BI_K}(\CX^{\Sigma}_{\CF_\ord}(\mathbf{f}/K_\infty))\right)\geq\ord_P(\CL^\Sigma_{\mathbf{f},K}),
\]
unless $P=P^+\BI[[\Gamma_K]]$ for some height one prime $P^+$ of $\BI[[\Gamma_K^+]]$.
\end{thm}

\begin{proof}
Let $\mathbf{D}=(A,\mathbf{f},1,1,\Sigma)$ be a $p$-adic Eisenstein datum as defined in \cite{skinner2014iwasawa}, with $A\supset\BZ_p[i,D_K]$ a finite $\BZ_p$-algebra. Define $\Lambda_{\mathbf{D}}:=\BI[[\Gamma_K]][[\Gamma_K^-]]$ as in \cite{skinner2014iwasawa}. The proof follows essentially from Section 7.4 of \cite{skinner2014iwasawa}, with the exception that we exclude primes $P=P^+\Lambda_{\mathbf{D}}$ for some height one prime $P^+$ of $\BI[[\Gamma_K^+]]$. 

By \cite[Proposition 13.6(1)]{skinner2014iwasawa}, for sufficiently large $\Sigma$, there exists a $p$-adic Eisenstein series $\mathbf{E}_{\mathbf{D}}$ whose formal $q$-expansion coefficients lie in $\Lambda_{\mathbf{D}}$, along with a set $\CC_{\mathbf{D}}$ of coefficients of $\mathbf{E}_{\mathbf{D}}$. This set satisfies that any height one prime $P$ of $\Lambda_{\mathbf{D}}$ containing $\CC_{\mathbf{D}}$ must be of the form $P= P^+\Lambda_{\mathbf{D}}$ for a height one prime $P^+$ of $\BI[[\Gamma_K^+]]$.

By \cite[Theorem 7.7]{skinner2014iwasawa}, we have
\[
\ord_P\left(\Char_{\BI_K}(\CX^{\Sigma}_{\CF_\ord}(\mathbf{f}/K_\infty))\right)\geq\ord_P(\CL^\Sigma_{\mathbf{f},K})
\]
for all height one primes $P$ of $\Lambda_{\mathbf{D}}$ satisfying
\begin{enumerate}
    \item $\mathbf{E}_{\mathbf{D}}$ is non-zero modulo $P$.
    \item $\ord_P(\CL_1^\Sigma)=0$, where $\CL_1^\Sigma\in \BI[[\Gamma_K^+]]$ denotes the auxiliary $p$-adic L-function defined in \cite[Section 6.5.3]{skinner2014iwasawa} (in fact, it is  the image of the $p$-adic L-function of the trivial Dirichlet character under the composition of the identification $A[[\Gamma_\BQ]]\simeq A[[\Gamma_K^+]]$ and the map $A[[\Gamma_K^+]]\ra \BI[[\Gamma_K]], \gamma^+\mapsto (1+W)^{-1} (\gamma^{+})^2$).
\end{enumerate}
Consequently, for all height one primes $P$ not of the form $P^+\Lambda_{\mathbf{D}}$,
\[
\ord_P\left(\Char_{\BI_K}(\CX^{\Sigma}_{\CF_\ord}(\mathbf{f}/K_\infty))\right)\geq\ord_P(\CL^\Sigma_{\mathbf{f},K}).
\]
\end{proof}

\begin{remark}
Since primes of the form $P= P^+\Lambda_{\mathbf{D}}$ for height one primes $P^+$ of $\BI[[\Gamma_K^+]]$ are excluded, Proposition 13.6 (2) of \cite{skinner2014iwasawa} is unnecessary. Hence, we omit the conditions on $N$ and on $\bar{\rho}_\mathbf{f}|_{I_\ell}$ for primes $\ell\mid N$, where $\bar{\rho}_\mathbf{f}:=\rho_\mathbf{f} \mod \fm_\BI$, with $\fm_\BI$ being the maximal ideal of $\BI$. See also the remark after Theorem 3.26 in \cite{skinner2014iwasawa}.
\end{remark}

\begin{cor}\label{BDP1}
Assume that the residual representation $\ov{\rho}_E|_{G_K}$ is irreducible.
There exists a nontrivial multiplicative set $S\subset \Lambda_K^+\subset\Lambda_K$ such that
\[
S^{-1}\Char_{\Lambda_K}(\CX_{\CF_\ord}(E/K_\infty))\subset(\CL_p^\mathrm{PR}(E/K))
\]
holds in $S^{-1}\Lambda_K$.
\end{cor}

\begin{proof}
By \cite[Corollary 1.5]{Hida86}, there exist  an $\BI$-adic ordinary eigenform  $\mathbf{f}$ with trivial nebentypus and a weight 2 arithmetic specialization $\phi\in\fX_{\BI,\CO}^a$ such that   the specialized form $\phi(\mathbf{f})$ is the $p$-stabilization of the newform associated to $E$.
Let $\fp_\phi$ be the kernel of the specialization map $\phi$.

By Propositions \ref{imp0} and  \ref{non1}, the $p$-adic $L$-function $\CL^\Sigma_{\mathbf{f},K}$ is not contained in $\fp_{\phi}\BI_K$. Let $\left\{ P_1,\dots,P_n\right\}$ be the set of all height one primes of $\BI_K$ such that  $P_i=P_i^+\BI_K$ for some height one prime $P_i^+$ of $\BI[[\Gamma_K^+]]$ and $\ord_{P_i}(\CL^\Sigma_{\mathbf{f},K})>0$. 
It follows that $P_i\not\subset \fp_{\phi}\BI_K$ for all $i$. 
For each $i$, choose an element $h_i \in P_i - \fp_{\phi}\BI_K$, and let $T$ be the multiplicative set generated by $\{h_i : i=1,\dots,n\}$. 
Then for any height one prime $P$ of $\BI_K$ with $P\cap T=\emptyset$, either $\ord_{P}(\CL^\Sigma_{\mathbf{f},K})=0$ or $P$ is not of the form $P^+\BI_K$ for any height one prime $P^+$ of $\BI[[\Gamma_K^+]]$. 
Hence by Theorem \ref{SU1}, we have that 
\[
T^{-1}\Char_{\BI_K}(\CX^{\Sigma}_{\CF_\ord}(\mathbf{f}/K_\infty))\subset(\CL^\Sigma_{\mathbf{f},K}).
\]
 Set $S=\phi(T)$. The result then follows from applying \cite[Corollary 3.8 (ii)]{skinner2014iwasawa}, Corollary \ref{Hida_des}, Corollary \ref{impr}, and Proposition \ref{imp0}.
\end{proof}

\subsection{Proof of  Theorem \ref{mainthm1}}\label{proof}

\begin{thm}\label{2_var}
For every nontrivial multiplicative set $S \subset \Lambda_K$, the following statements are equivalent:
\begin{enumerate}
    \item $S^{-1}\Char_{\Lambda_K}(\CX_{\CF_\ord}(E/K_\infty)) \subset (\CL_p^\mathrm{PR}(E/K))$.
    \item $S^{-1}\Char_{\Lambda_K}(\CX_{\CF_\Gr}(E/K_\infty))\Lambda_K^{\ur} \subset (\CL_p^{\Gr}(E/K))$.
\end{enumerate}
The same equivalence also holds for the reverse inclusions.
\end{thm}

\begin{proof}
This result is essentially \cite[Proposition 9.18]{BSTW} in the two-variable (or $\cdot=\emptyset$ in \emph{loc.\ cit.}) case, building on explicit reciprocity laws for the Beilinson-Flach classes and global Poitou-Tate duality. 
See also \cite[Proposition 4.1.3]{burungale2024base} and \cite[Proposition 4.2.1]{CGS}.
\end{proof}

By Theorem \ref{2_var} and Corollary \ref{BDP1}, there exists a nontrivial multiplicative set $S\subset \Lambda_K^+\subset\Lambda_K$ such that
\[
S^{-1}\Char_{\Lambda_K}(\CX_{\CF_\Gr}(E/K_\infty))\Lambda_K^{\ur}\subset(\CL_p^{\Gr}(E/K)).
\]
We may assume $S$ is generated by prime elements in the unique factorization domain $\Lambda_K$. For a height-one prime $P\subset \Lambda_K^{\ur}$ with $P\cap S\neq \emptyset$, we have $P=P^+\Lambda_K^{\ur}$ for some prime $P^+\subset \Lambda_K^{\ur,+}$. However, by \cite[Theorem B]{hsieh2014special} and Proposition \ref{comp},
\[
\mu(\CL_p^{\Gr}(E/K)^-)=\mu(\CL_p^{\mathrm{BDP}}(E/K))=0,
\]
where $\mu(\cdot)$ denotes the $\mu$-invariant. Hence, we have
\[
\ord_{P}(\CL_p^{\Gr}(E/K))=0
\]
if $P\subset \Lambda_K^{\ur}$ is a height-one prime of the form $P=P^+\Lambda_K^{\ur}$ for some $P^+\subset \Lambda_K^{\ur,+}$. It follows that
\[
\Char_{\Lambda_K}(\CX_{\CF_\Gr}(E/K_\infty))\Lambda_K^{\ur}\subset(\CL_p^{\Gr}(E/K)).
\]

Moreover, if condition (\ref{Imag}) holds, we can establish that part (1) of Conjecture \ref{mainconj} is true. The argument is analogous to that of \cite[Theorem 3.30]{skinner2014iwasawa}, employing  the results of Kato in \cite[Theorem 17.4]{Kato}, Lemma \ref{sum}, and a commutative algebra lemma from \cite[Lemma 3.2]{skinner2014iwasawa}. The validity of part (1) of the conjecture, in turn, implies that part (2) also holds.
We note that while the original proof in \cite{skinner2014iwasawa} assumes the stronger condition $\Im \rho_{E}\supset \SL_2(\BZ_p)$, this requirement can be relaxed to our condition (\ref{Imag}), as is discussed in the final paragraph of \cite[page 187]{ski2016}.

\section{Applications}
\subsection{Cyclotomic main conjectures}

Let $E/\BQ$ be an elliptic curve of conductor $N$, and let $p>2$ be a prime at which $E$ has good ordinary reduction. 
Assume that the residual representation $\bar{\rho}_{E}: G_\BQ \ra \Aut(E[p])$ is irreducible.  

\begin{conj}[Mazur's main conjecture]\label{cyc_main}
    $\CX_{\CF_\ord}(E/\BQ_\infty)$ is $\Lambda_\BQ$-torsion and
    \[\Char_{\Lambda_\BQ}(\CX_{\CF_\ord}(E/\BQ_\infty))=(\CL_p^{\mathrm{MSD}}(E/\BQ)).\]
\end{conj}

As in \cite[Theorem 3.29]{skinner2014iwasawa}, we have the following result. 

\begin{thm}\label{cyc}
    $\CX_{\CF_\ord}(E/\BQ_\infty)$ is $\Lambda_\BQ$-torsion and
    \[
    \Char_{\Lambda_\BQ}(\CX_{\CF_\ord}(E/\BQ_\infty))\otimes\BQ_p=(\CL_p^{\mathrm{MSD}}(E/\BQ))
    \]
    in $\Lambda_\BQ\otimes \BQ_p$. 
    Moreover, if condition (\ref{Imag}) holds, we have
    \[
    \Char_{\Lambda_\BQ}(\CX_{\CF_\ord}(E/\BQ_\infty))=(\CL_p^{\mathrm{MSD}}(E/\BQ))
    \]
    in $\Lambda_\BQ$.
\end{thm}

\begin{lem}
    Let $K$ be an imaginary quadratic field, and $I\subset \Lambda_K$  the kernel of the natural projection ${\Lambda_K\ra \Lambda^+_K}$.
   Suppose that the residual representation $\bar{\rho}_{E}|_{G_K}: G_K \ra \Aut(E[p])$ is irreducible. 
   Then 
   \[ \Char_{\Lambda_\BQ}(\CX_{\CF_\ord}(E/\BQ_\infty)) \cdot \Char_{\Lambda_\BQ}(\CX_{\CF_\ord}(E^K/\BQ_\infty))\subset \Char_{\Lambda_K}(\CX_{\CF_\ord}(E/K_\infty)) \mod  I\]
   in $\Lambda_K^+\simeq \Lambda_\BQ$.
\end{lem}
\begin{proof}
        Let $\Sigma$ be a set of places of $K$ containing all places dividing $pN\infty$ as before, $\Sigma_0:=\{\fq\in \Sigma:\fq\nmid p\}$.
    Let $M:=T_pE\otimes \Lambda_K^\vee$ and $N:=T_pE\otimes \Lambda_K^{+,\vee}$ for simplicity.

    We have the following commutative diagram of exact sequences:
    \[\begin{tikzcd}
0 \arrow[r] & {H^1_{\mathcal{F}_{\mathrm{ord}}}(K, N)} \arrow[r] \arrow[d, "f"] & {H^1_{\mathcal{F}_{\mathrm{ord}}^\Sigma}(K, N)} \arrow[r, "s"] \arrow[d, "g"] & {\oplus_{\fq\in\Sigma_0}H^1(G_{K_\mathfrak{q}},N)} \arrow[d, "h"] \\
0 \arrow[r] & {H^1_{\mathcal{F}_{\mathrm{ord}}}(K, M)[I]} \arrow[r]                & {H^1_{\mathcal{F}_{\mathrm{ord}}^\Sigma}(K, M)[I]} \arrow[r]                & {\oplus_{\fq\in \Sigma_0}H^1(G_{K_\mathfrak{q}},M)[I]}               
\end{tikzcd},
\]
where $f,g,h,s$ are the natural maps.  We claim that $h$ is injective. To see this, for any primes $\fq \in \Sigma_0$, we consider the following commutative diagram of inflation-restriction sequences
 \[\begin{tikzcd}
0 \arrow[r] & {H^1(G_{K_\fq}/I_\fq, N^{I_\fq})} \arrow[r] \arrow[d] & {H^1(G_{K_\fq}, N)} \arrow[r] \arrow[d, "h_\fq"] & {H^1(I_\fq,N)} \arrow[d, "t_\fq"] \\
0 \arrow[r] & {H^1(G_{K_\fq/}I_\fq, M^{I_\fq})[I]} \arrow[r]                & {H^1(G_{K_\fq}, M)[I]} \arrow[r]                & {H^1(I_\fq,M)[I]}               
\end{tikzcd},
\]
where $h_\fq, t_\fq$ are the natural maps. Since $H^1(G_{K_\fq}/I_\fq, N^{I_\fq})=0$ by Lemma \ref{loc} (1), to prove that $h_\fq$ is injective, it suffices to show that $t_\fq$ is injective. Fix a topological generator $\gamma^-$ of $\Gal(K_\infty^-/K)\simeq \Gal(K_\infty/K_\infty^+)$, then we have $I=(\gamma^--1)\subset \Lambda_K$ as an ideal.
Then we have the long exact sequence
\[H^0(I_\fq, M)\xrightarrow{\cdot (\gamma^--1)}H^0(I_\fq, M)\ra H^1(I_\fq, N)\xrightarrow{t_\fq} H^1(I_\fq, M)\]
induced by the short exact sequence
\[0\ra N\ra M\xrightarrow{\cdot (\gamma^--1)}M\ra 0,\]
hence to prove the injectivity of $t_\fq$, it suffices to show $H^0({I_\fq,M})/I H^0({I_\fq,M})=0$. However, since $\fq$ is unramified in $K_\infty/K$, let $W:=T_p\otimes \BQ_p/\BZ_p$, then we have that
\[H^0({I_\fq,M})= (T_pE\otimes \Hom_{\mathrm{cont}}(\Lambda_K,\BQ_p/\BZ_p))^{I_\fq}\simeq \Hom_{\mathrm{cont}}(\Lambda_K,W)^{I_\fq}=\Hom_{\mathrm{cont}}(\Lambda_K,W^{I_\fq}).\]
Since
\[\begin{aligned}
    \Hom_{\mathrm{cont}}(\Lambda_K,W^{I_\fq})&\simeq \Hom_{\mathrm{cont}}(\Lambda_K,\Hom(\Hom(W^{I_\fq},\BQ_p/\BZ_p),\BQ_p/\BZ_p))\\
    &\simeq \Hom_{\mathrm{cont}}(\Hom(W^{I_\fq},\BQ_p/\BZ_p)\otimes \Lambda_K,\BQ_p/\BZ_p),
\end{aligned}\]
we have
\[\Hom(\Hom_{\mathrm{cont}}(\Lambda_K,W^{I_\fq}),\BQ_p/\BZ_p)\simeq \Hom(W^{I_\fq},\BQ_p/\BZ_p)\otimes \Lambda_K.\]

Hence
\[\Hom(H^0({I_\fq,M})/I H^0({I_\fq,M}),\BQ_p/\BZ_p)\simeq \Hom(W^{I_\fq},\BQ_p/\BZ_p)\otimes \Lambda_K[I]=0.\]
Therefore $H^0({I_\fq,M})/I H^0({I_\fq,M})=0$, which completes the proof of our claim.

Let $A:=\Im s \subset \oplus_{\fq\in\Sigma_0}H^1(G_{K_\mathfrak{q}},N)$.
Then by the snake lemma,  we have the following exact sequence
\[0\ra \ker f\ra \ker g\ra \ker h|_A\ra \coker f\ra \coker g.\]
By \cite[Proposition 3.9]{skinner2014iwasawa}, Remark \ref{coin},  Lemma \ref{trivial} and Nakayama's lemma, we have that the map  $g$  is an  isomorphism. 
Hence we have $\ker f=0$ and $\coker f \simeq  \ker h|_A=0$, which implies that the map $f$ is an isomorphism. Then we have  that the natural $\Lambda_K^+$-module homomorphism 
    \[\mathcal{X}_{\mathcal{F}_\ord}(E/K_\infty)/I \mathcal{X}_{\mathcal{F}_\ord}(E/K_\infty)\rightarrow\mathcal{X}_{\mathcal{F}_\ord}(E/K_\infty^+)\]
    is an isomorphism. 
    By \cite[Corollary 3.8 (ii)]{skinner2014iwasawa}, we have that 
     \[ \Char_{\Lambda_K^+}(\CX_{\CF_\ord}(E/K_\infty^+)) \subset \Char_{\Lambda_K}(\CX_{\CF_\ord}(E/K_\infty)) \mod  I\]
   in $\Lambda_K^+$. Together with  Lemma \ref{sum}, we complete the proof.
\end{proof}

\begin{proof}[Proof of Theorem \ref{cyc}]
    By a result of Serre (\cite[Section 3.3]{Serre}), the irreducibility of $\bar{\rho}_{E}$ implies that it is absolutely irreducible when $p>2$.
     Choose an imaginary quadratic field $K$ such that all primes dividing $pN$ split in $K$ and  the restriction $\bar{\rho}_{E}|_{G_K}$ remains absolutely irreducible. Let $I\subset \Lambda_K$ be the kernel of the natural projection ${\Lambda_K\ra \Lambda^+_K}$. Applying Theorem \ref{mainthm1}, Proposition \ref{PR-MSD}, and the preceding lemma, we obtain
    \begin{align*}
        &\Char_{\Lambda_\BQ}(\CX_{\CF_\ord}(E/\BQ_\infty)) \cdot \Char_{\Lambda_\BQ}(\CX_{\CF_\ord}(E^K/\BQ_\infty))\subset \Char_{\Lambda_K}(\CX_{\CF_\ord}(E/K_\infty)) \mod  I\\
         &\subset (\CL_p^\mathrm{PR}(E/K)) \mod I = (\mathcal{L}_p^{\mathrm{PR}}(E/K)^+) = (\mathcal{L}_p^{\mathrm{MSD}}(E/\mathbb{Q}) )\cdot (\mathcal{L}_p^{\mathrm{MSD}}(E^K/\mathbb{Q}))
    \end{align*}
    in $\Lambda_\BQ\simeq \Lambda_K^+$.
    On the other hand, by \cite[Theorem 17.14]{Kato}, we have that
    \[ (\mathcal{L}_p^{\mathrm{MSD}}(E/\mathbb{Q}))\subset \Char_{\Lambda_\BQ}(\CX_{\CF_\ord}(E/\BQ_\infty)), \quad (\mathcal{L}_p^{\mathrm{MSD}}(E^K/\mathbb{Q}))\subset\Char_{\Lambda_\BQ}(\CX_{\CF_\ord}(E^K/\BQ_\infty))\]
    in $\Lambda_\BQ\otimes \BQ_p$ and even in $\Lambda_\BQ$ under the condition (\ref{Imag}) (See the discussions in the final paragraph of \cite[page 187]{ski2016}, as noted previously). 
    If one of the above inclusions is not an equality, then 
       \[(\mathcal{L}_p^{\mathrm{MSD}}(E/\mathbb{Q})) \cdot(\mathcal{L}_p^{\mathrm{MSD}}(E^K/\mathbb{Q}))\subsetneq \Char_{\Lambda_\BQ}(\CX_{\CF_\ord}(E/\BQ_\infty)) \cdot\Char_{\Lambda_\BQ}(\CX_{\CF_\ord}(E^K/\BQ_\infty)),\]
       which contradicts the inclusion established above:
       \[\Char_{\Lambda_\BQ}(\CX_{\CF_\ord}(E/\BQ_\infty)) \cdot \Char_{\Lambda_\BQ}(\CX_{\CF_\ord}(E^K/\BQ_\infty))\subset (\mathcal{L}_p^{\mathrm{MSD}}(E/\mathbb{Q}) )\cdot (\mathcal{L}_p^{\mathrm{MSD}}(E^K/\mathbb{Q}))\]
       
     Therefore it must be that
       \[
    \Char_{\Lambda_\BQ}(\CX_{\CF_\ord}(E/\BQ_\infty))=(\CL_p^{\mathrm{MSD}}(E/\BQ))
    \]
    in $\Lambda_\BQ\otimes \BQ_p$ and even in $\Lambda_\BQ$ under the condition (\ref{Imag}).
\end{proof}

As a corollary, we obtain the following result concerning the cyclotomic analogue of the BDP main conjecture.

\begin{cor}\label{cyc-BDP}
    If $\CL_p^{\mathrm{Gr}}(E/K)^+$ is nontrivial, then
    $\CX_{\CF_\Gr}(E/K_\infty^+)$ is $\Lambda_K^+$-torsion and
    \[
    \Char_{\Lambda_K^+}(\CX_{\CF_\Gr}(E/K_\infty^+))\Lambda_K^{\ur,+}\otimes\BQ_p=(\CL_p^{\mathrm{Gr}}(E/K)^+)
    \]
    in $\Lambda_K^{\ur,+}\otimes \BQ_p$. 
    Moreover, if condition (\ref{Imag}) holds, we have
    \[
    \Char_{\Lambda_K^+}(\CX_{\CF_\Gr}(E/K_\infty^+))\Lambda_K^{\ur,+}=(\CL_p^{\mathrm{Gr}}(E/K)^+)
    \]
    in $\Lambda_K^{\ur,+}$.
\end{cor}

\begin{proof}
    This result follows from a combination of Theorem \ref{cyc} (applied to $E$ and $E^K$),  Lemma \ref{sum}, Proposition \ref{PR-MSD}, \cite[Proposition 4.2.1]{CGS} (a cyclotomic analogue of Theorem \ref{2_var}, see also \cite[Proposition 9.18]{BSTW}), and Proposition \ref{non1}.
\end{proof}

\subsection{Anticyclotomic main conjectures}

Let $E/\BQ$ be an elliptic curve of conductor $N$, let $p>2$ be a prime such that $E$ has good ordinary reduction at $p$, and let $K$ be an imaginary quadratic field such that $p\CO_K = \fp \bar{\fp}$ splits in $K$ and the pair $(E, K)$ satisfies the Heegner hypothesis. 
Assume that the residual representation $\bar{\rho}_E|_{G_K}: G_K \ra \Aut(E[p])$ is irreducible.

\begin{conj}[BDP main conjecture] 
     $\CX_{\CF_\Gr}(E/K_\infty^-)$ is $\Lambda_K^-$-torsion and
    \[
    \Char_{\Lambda_K^-}(\CX_{\CF_\Gr}(E/K_\infty^-))\Lambda_K^{\ur,-}=(\CL_p^{\mathrm{BDP}}(E/K)).
    \]    
\end{conj}

We define the compact Selmer group
\[
S_{\ord}(E/K_\infty^-) = \varprojlim_{n}~\varprojlim_{m}~\Sel_{p^m}(E/K_n^-),
\]
where $K_n^-$ is the subfield of $K_\infty^-$ with $[K_n^-:K]=p^n$ for each positive integer $n$.

Fix a modular parametrization $\pi: X_0(N)\rightarrow E$. 
In \cite{PR}, Perrin-Riou constructed an element $\kappa \in S_{\ord}(E/K_\infty^-)$ using the Kummer images of Heegner points on $X_0(N)$, which is non-torsion over $\Lambda_K^-$ by a result of Cornut-Vatsal \cite{CV}. 
She then formulated the anticyclotomic main conjecture as follows.

\begin{conj}[Heegner point main conjecture]
    The modules $S_{\ord}(E/K_\infty^-)$ and $\CX_{\CF_\ord}(E/K_\infty^-)$ are both of $\Lambda_K^-$-rank one, and
    \[
    \Char_{\Lambda_K^-}(\CX_{\CF_\ord}(E/K_\infty^-)_\tor)
    =\Char_{\Lambda_K^-}(S_{\ord}(E/K_\infty^-)/\Lambda_K^-\cdot\kappa)^2.
    \] 
\end{conj}

\begin{thm}\label{ac}
    \begin{enumerate}
        \item $\CX_{\CF_\Gr}(E/K_\infty^-)$ is $\Lambda_K^-$-torsion and
        \[
        \Char_{\Lambda_K^-}(\CX_{\CF_\Gr}(E/K_\infty^-))\Lambda_K^{\ur,-}\otimes\BQ_p=(\CL_p^{\mathrm{BDP}}(E/K))
        \]
        holds in $\Lambda_K^{\ur,-} \otimes \BQ_p$. 
        Moreover, if the representation $\rho_E|_{G_K}: G_K \ra \Aut_{\BZ_p}(T_pE)$ is surjective, then 
        \[
        \Char_{\Lambda_K^-}(\CX_{\CF_\Gr}(E/K_\infty^-))\Lambda_K^{\ur,-}=(\CL_p^{\mathrm{BDP}}(E/K)).
        \]

        \item The modules $S_{\ord}(E/K_\infty^-)$ and $\CX_{\CF_\ord}(E/K_\infty^-)$ are both of $\Lambda_K^-$-rank one, and
        \[
        \Char_{\Lambda_K^-}(\CX_{\CF_\ord}(E/K_\infty^-)_\tor)\otimes\BQ_p
        =\Char_{\Lambda_K^-}(S_{\ord}(E/K_\infty^-)/\Lambda_K^-\cdot\kappa)^2 \otimes\BQ_p
        \]
        holds in $\Lambda_K^- \otimes \BQ_p$. 
        Moreover, if the representation $\rho_E|_{G_K}: G_K \ra \Aut_{\BZ_p}(T_pE)$ is surjective, then
        \[
        \Char_{\Lambda_K^-}(\CX_{\CF_\ord}(E/K_\infty^-)_\tor)
        =\Char_{\Lambda_K^-}(S_{\ord}(E/K_\infty^-)/\Lambda_K^-\cdot\kappa)^2.
        \]
    \end{enumerate}
\end{thm}

We prove this theorem in the remainder of this subsection. 
We begin by recalling some results concerning the Heegner point main conjecture. 

\begin{thm}\label{HPMC}
    Assume that $E(K)[p]=0$.
    Then the modules $S_{\ord}(E/K_\infty^-)$ and $\CX_{\CF_\ord}(E/K_\infty^-)$ are both of $\Lambda_K^-$-rank one, and
    \[
    \Char_{\Lambda_K^-}(\CX_{\CF_\ord}(E/K_\infty^-)_\tor)\otimes\BQ_p
    \supset
    \Char_{\Lambda_K^-}(S_{\ord}(E/K_\infty^-)/\Lambda_K^-\cdot\kappa)^2\otimes\BQ_p
    \]
    holds in $\Lambda_K^-\otimes\BQ_p$. 
    Moreover, if the representation $\rho_E|_{G_K}: G_K\rightarrow\Aut_{\BZ_p}(T_pE)$ is surjective, then
    \[
    \Char_{\Lambda_K^-}(\CX_{\CF_\ord}(E/K_\infty^-)_\tor)
    \supset 
    \Char_{\Lambda_K^-}(S_{\ord}(E/K_\infty^-)/\Lambda_K^-\cdot\kappa)^2.
    \]
\end{thm}

\begin{proof}
   By \cite[Theorem 6.5.2]{CGS}, the modules $S_{\ord}(E/K_\infty^-)$ and $\CX_{\CF_\ord}(E/K_\infty^-)$ both have $\Lambda_K^-$-rank one, and the inclusion
    \[
    \Char_{\Lambda_K^-}(\CX_{\CF_\ord}(E/K_\infty^-)_\tor)\otimes\BQ_p
    \supset
    \Char_{\Lambda_K^-}(S_{\ord}(E/K_\infty^-)/\Lambda_K^-\cdot\kappa)^2\otimes\BQ_p
    \]
    holds in $\Lambda_K^-\otimes\BQ_p$. Moreover, if the representation $\rho_E|_{G_K}: G_K\rightarrow\Aut_{\BZ_p}(T_pE)$ is surjective, then by \cite[Theorem B]{How}, the inclusion
    \[
    \Char_{\Lambda_K^-}(\CX_{\CF_\ord}(E/K_\infty^-)_\tor)
    \supset 
    \Char_{\Lambda_K^-}(S_{\ord}(E/K_\infty^-)/\Lambda_K^-\cdot\kappa)^2
    \]
    holds in $\Lambda_K^-$ (i.e., without inverting $p$).
\end{proof}

Arguing as in \cite[Theorem 5.2]{BCK}, we obtain the following result.

\begin{thm}\label{ant}
    The module $\CX_{\CF_\Gr}(E/K_\infty^-)$ is $\Lambda_K^-$-torsion, and for every nontrivial multiplicative set $S\subset \Lambda_K^-$, the following statements are equivalent:
    \begin{enumerate}
        \item $S^{-1}\Char_{\Lambda_K^-}(\CX_{\CF_\Gr}(E/K_\infty^-))\Lambda_K^{\ur,-} \supset (\CL_p^{\mathrm{BDP}}(E/K))$.

        \item $S^{-1}\Char_{\Lambda_K^-}(\CX_{\CF_\ord}(E/K_\infty^-)_\tor) \supset 
        S^{-1}\Char_{\Lambda_K^-}(S_{\ord}(E/K_\infty^-)/\Lambda_K^-\cdot\kappa)^2$.
    \end{enumerate}
    The equivalence also holds with the divisibility reversed.
\end{thm}

Since $\CL_p^{\mathrm{BDP}}(E/K)$ is non-zero (\cite[Corollary 4.5]{BCK}), by the arguments presented in Section \ref{proof}, there exists a nontrivial multiplicative set $S\subset \Lambda_K^+ \subset \Lambda_K$ such that for any $s\in S$, $\gamma^+-1\nmid s$, and 
\[
S^{-1}\Char_{\Lambda_K}(\CX_{\CF_\Gr}(E/K_\infty))\Lambda_K^{\ur} \subset (\CL_p^{\Gr}(E/K)).
\]
By Lemma \ref{bdp_des} and standard properties of characteristic ideals (see, e.g., \cite[Corollary 3.8]{skinner2014iwasawa}), we have
\[
\Char_{\Lambda_K^-}(\CX_{\CF_\Gr}(E/K_\infty^-))\Lambda_K^{\ur,-} \otimes\BQ_p \subset (\CL_p^{\mathrm{BDP}}(E/K)).
\]
The reverse divisibility is established by combining Theorem \ref{HPMC} and Theorem \ref{ant}.

Moreover, if the representation $\rho_E|_{G_K}: G_K \ra \Aut_{\BZ_p}(T_pE)$ is surjective, then by Theorem \ref{mainthm1} and Lemma \ref{bdp_des}, 
\[
\Char_{\Lambda_K^-}(\CX_{\CF_\Gr}(E/K_\infty^-))\Lambda_K^{\ur,-} \subset (\CL_p^{\mathrm{BDP}}(E/K)),
\]
and the rest of the argument follows from Theorem \ref{HPMC} and Theorem \ref{ant} as before.

\begin{remark}
    Via \cite[Proposition 12.7]{BSTW}, the condition in Theorem \ref{ac} that $\rho_E|_{G_K}$ has full image can be weakened to condition (\ref{Imag}).
\end{remark}

\subsection{BSD conjectures}

Let $E/\BQ$ be an elliptic curve of conductor $N$, and let $p > 2$ be a prime such that $E$ has good ordinary reduction at $p$.

\begin{thm}
Let $r \leq 1$ be an integer. The following statements are equivalent:
\begin{enumerate}
    \item $\mathrm{rank}_\BZ E(\BQ) = r$ and $\#\Sha(E/\BQ) < \infty$;
    \item $\mathrm{corank}_{\BZ_p} \Sel_{p^\infty}(E/\BQ) = r$;
    \item $\ord_{s=1} L(E/\BQ, s) = r$.
\end{enumerate}
Under any of the above conditions, if condition (\ref{Imag}) also holds, then the $p$-part of the BSD formula for $E$ is valid, i.e.,
\[
\left| \frac{L^{(r)}(1, E)}{r! \cdot \Omega_E R_E} \right|_{p} 
= \left| \frac{\#\Sha(E)[p^\infty] \cdot \prod_{\ell \mid N} c_\ell(E)}{(\#E(\BQ)_\tor)^2} \right|_{p},
\]
where $R_E$ is the regulator of $E(\BQ)$, $\Omega_E$ is the \Neron period, $c_\ell(E)$ is the Tamagawa number at a prime $\ell$, and $|\cdot|_p$ denotes the $p$-adic absolute value.
\end{thm}

\begin{proof}
The implication (1) $\Rightarrow$ (2) is immediate.

The implication (3) $\Rightarrow$ (1) is established by the work of Gross-Zagier \cite{GZ} and Kolyvagin \cite{Kol}.

For the implication (2) $\Rightarrow$ (3), one can choose an imaginary quadratic field $K$ such that $\ord_{s=1} L(E^{K}, s) \leq 1$ and $(E, K)$ satisfies the Heegner hypothesis, as shown in \cite{FH}. Analogously to \cite[Theorem 1.9]{Wan}, by applying descent arguments to the rational part of Theorem \ref{ac}(2) and using the Gross-Zagier formula, we obtain
\[
\mathrm{corank}_{\BZ_p} \Sel_{p^\infty}(E/K) = 1 \quad \Rightarrow \quad \ord_{s=1} L(E/K, s) = 1,
\]
thus establishing (2) $\Rightarrow$ (3).

The $p$-part of the BSD formula in the rank zero case follows from the integral part of Theorem \ref{cyc}, combined with descent arguments (see \cite[Section 3.6.1]{skinner2014iwasawa} for details).

Similarly, by choosing an imaginary quadratic field $K$ as described above, the $p$-part of the BSD formula in the rank one case follows from the integral part of Corollary \ref{cyc-BDP}, \cite[Proposition 3.4.2]{CGS}, the fact that $\CL_p^{\mathrm{Gr}}(E/K)^+(1) = \CL_p^{\mathrm{Gr}}(E/K)^-(1) \neq 0$, and the descent arguments in \cite{JSW}.
\end{proof}

\section*{Conflict of interest}
On behalf of all authors, the corresponding author states that there is no conflict of interest.

\bibliographystyle{amsalpha}
\bibliography{ref}
\end{document}